\documentclass[10pt]{amsart} 

\usepackage{amssymb,amsmath}
\usepackage{times}
\usepackage{amscd}

\usepackage{graphicx}
\usepackage{pinlabel}
\usepackage{color}

\newtheorem{thm}{Theorem}[section]
\newtheorem{cor}[thm]{Corollary}
\newtheorem{defin}[thm]{Definition}
\newtheorem{rmk}[thm]{Remark}
\newtheorem{lem}[thm]{Lemma}
\newtheorem{prop}[thm]{Proposition}

\newcommand{\comments}[1]{}
\newcommand{\mcD}{\mbox{$\mathcal{D}$}}

\newcommand{\mcP}{\mbox{$\mathcal{P}$}}

\title{High Distance Heegaard Splittings via Dehn Twists}

\begin{document}
\author{Michael Yoshizawa}
\date{\today}

\thanks{We would like to thank Martin Scharlemann for his guidance and many helpful conversations.  Special thanks also to Tsuyoshi Kobayashi and Yo'av Rieck for their insightful observations and the math department of Nara Women's University for its hospitality.  The author was supported in part by the NSF grants OISE-1209236 and DMS-1005661.}

\begin{abstract}
In 2001, J. Hempel proved the existence of Heegaard splittings of arbitrarily high distance by using a high power of a pseudo-Anosov map as the gluing map between two handlebodies.  We show that lower bounds on distance can also be obtained when using a high power of a suitably chosen Dehn twist.  In certain cases, we can then determine the exact distance of the resulting splitting.  These results can be seen as a natural extension of work by A. Casson and C. Gordon in 1987 regarding strongly irreducible Heegaard splittings.
\end{abstract}

\maketitle

\section{Introduction}

Since Hempel \cite{He} introduced the notion of the distance of a Heegaard splitting in 2001, there have been a number of results linking distance with the topology of the ambient 3-manifold.  For example, work by Thompson \cite{Th} and Hempel \cite{He} showed that if a manifold admits a Heegaard splitting of distance $\geq 3$, then the manifold is atoroidal and not Seifert fibered, hence hyperbolic by the Geometrization Conjecture.  Hartshorn \cite{Ha} found that if a manifold admits a distance $d$ Heegaard splitting, then the genus $g$ of an orientable incompressible surface must satisfy $2g \geq d$.  Scharlemann and Tomova \cite{ST} generalized this result to show that a Heegaard splitting of genus $g$ and distance $d$ such that $2g < d$ is the unique splitting of minimal genus for that manifold.

Concurrently, there have been many efforts to construct examples of Heegaard splittings of high distance.  Many approaches make use of the fact that a Heegaard splitting can be described by a homeomorphism between two handlebodies of equal genus.  One of the first major achievements was made by Hempel \cite{He} who, adapting an argument of Kobayashi \cite{Ko2}, proved the existence of Heegaard splittings of arbitrarily high distance via the use of a high power of a pseudo-Anosov map as the gluing map between two handlebodies.  However, the use of pseudo-Anosov maps meant no concrete examples were provided.


This can be rectified by considering Heegaard splittings determined by a Dehn twist map, rather than one that is pseudo-Anosov.  Casson and Gordon \cite{CG} introduced what is now referred to as the Casson-Gordon rectangle condition (published by Kobayashi in \cite{Ko}) that would ensure a Heegaard splitting has distance $\geq 2$.  Moreover, they provided methods of generating Heegaard splittings using Dehn twists that satisfied this condition.  Two of their results can be stated as the following.  In the following statements, $D(H)$ denotes the disk set of a handlebody $H$ and distance refers to the distance in the curve complex of $\partial H$ (these definitions and more details appear in Section \ref{sec:heegaard}).

\begin{thm} \label{thm:CG1}
\textnormal{(Casson-Gordon \cite{CG}).} Suppose $H$ is a genus $g$ handlebody.  Let $\gamma$ be a simple closed curve that is distance $\geq 2$ from $D(H)$.  Then gluing $H$ to a copy of itself via $\geq 2$ Dehn twists about $\gamma$ determines a genus $g$ Heegaard splitting of distance $\geq 2$.
\end{thm}

\begin{thm} \label{thm:CG2}
\textnormal{(Casson-Gordon \cite{CG}).} Suppose $H_1$ and $H_2$ are genus $g$ handlebodies with $\partial H_1 = \partial H_2$ and the distance between $D(H_1)$ and $D(H_2)$ is at most 1. Let $\gamma$ be a simple closed curve that is distance $\geq 2$ from both $D(H_1)$ and $D(H_2)$.  Then gluing $H_1$ to $H_2$ via $\geq 6$ Dehn twists about $\gamma$ determines a genus $g$ Heegaard splitting of distance $\geq 2$.
\end{thm}

More recently, Berge developed a modified rectangle condition that would guarantee a genus 2 splitting has distance $\geq 3$ (this criterion is described by Scharlemann in \cite{Sc}) and constructs examples of genus 2 splittings satisfying this condition. Hempel \cite{He} used Dehn twists and the notion of stacks to construct Heegaard splittings of distance $\geq 3$ and Evans \cite{Ev} extended this result to Heegaard splittings of distance $\geq d$ for any $d \geq 2$.  Lustig and Moriah \cite{LM} used Dehn twists and derived train tracks to produce another class of examples of splittings with distance $\geq d$.  However, all of these results only provided a lower bound on the distance of the constructed splittings; the exact distance of the constructed examples of high distance splittings were typically unknown.

In fact, until recently there was no proof that there existed Heegaard splittings with distance equal to $d$ for every $d \in \mathbb{N}$.  This question has largely been settled by the work of Ido, Jang, and Kobayashi \cite{IJK}, who have developed examples of high distance Heegaard splittings with an exact known distance. For any genus $g \geq 2$, they construct a genus $g$ Heegaard splitting of a closed 3-manifold of distance $d$ for any $d \geq n_g$, where $n_g$ is a constant solely dependent on $g$ and whose existence is due to a result of Masur and Minsky \cite{MM3} regarding the quasi-convexity of the disk complex.  

We also provide examples of Heegaard splittings with an exact known distance using a different approach that relies exclusively on Dehn twist maps. The benefit of this approach is we avoid the need for the distance of our examples to be at least $n_g$. Our results, shown below, can be seen as an extension of the work by Casson and Gordon (Theorems \ref{thm:CG1} and \ref{thm:CG2}).  They are proved in Section \ref{sec:main}.
 
\begin{thm}
Suppose $H$ is a genus $g$ handlebody.  Let $\gamma$ be a simple closed curve that is distance $d$ from $D(H)$ for some $d \geq 2$.  Then gluing $H$ to a copy of itself via $\geq 2d-2$ Dehn twists about $\gamma$ determines a genus $g$ Heegaard splitting of distance exactly $2d-2$.
\end{thm}

\begin{thm}
Suppose $H_1$ and $H_2$ are genus $g$ handlebodies with $\partial H_1 = \partial H_2$ and $n = \max \{1,d(D(H_1),D(H_2))\}$.  Let $\gamma$ be a simple closed curve that is distance $d_1$ from $D(H_1)$ and $d_2$ from $D(H_2)$ where $d_1 \geq 2$, $d_2 \geq 2$, and $d_1 + d_2 - 2 > n$.  Then gluing $H_1$ to $H_2$ via $\geq n + d_1 + d_2$ Dehn twists about $\gamma$ determines a genus $g$ Heegaard splitting whose distance is at least $d_1 + d_2 - 2$ and at most $d_1 + d_2$.
\end{thm}

As a separate application of the machinery developed to obtain the above results, we show in Section \ref{sec:evans} that the process described by Evans \cite{Ev} can be replicated with fewer hypotheses and have higher distance than originally proven.

\section{Standard Cut Systems and Pants Decompositions}

Throughout this paper, $\Sigma$ will denote a closed orientable surface with genus $g$ such that $g \geq 2$.

\begin{defin}
A \emph{standard cut system of $\Sigma$} is a collection of $g$ essential simple closed curves $X$ in $\Sigma$ such that $\Sigma - X$ is a $2g$-punctured sphere.
\end{defin}

\begin{defin}
Let $H$ be a genus $g$ handlebody.  A \emph{standard set of compressing disks of $H$} is a collection of disjoint compressing disks $\mcD \subset H$ so that $\partial \mcD$ is a standard cut system in $\partial H$.  A \emph{standard set of meridians of $H$} is any collection of curves in $\partial H$ that bounds a standard set of compressing disks of $H$.  (Thus a standard set of meridians in $\partial H$ is a standard cut system, but not vice versa.)
\end{defin}

\begin{defin}
A \emph{pants decomposition} $\mcP$ of $\Sigma$ is a collection of $3g-3$ essential simple closed curves in $\Sigma$ such that $\Sigma - \mcP$ is a collection of pairs of pants (i.e. three punctured spheres).
\end{defin}

\begin{defin}
A \emph{complete collection of compressing disks} for a handlebody $H$ is a collection of disjoint compressing disks whose boundary is a pants decomposition of $\partial H$. The boundary of such a complete collection of disks is also called a pants decomposition of $H$.  
\end{defin}

\begin{defin}
Let $\mcP$ be a pants decomposition of $\Sigma$.  If $P$ is the closure of a pair of pants component of $\Sigma - \mcP$:
\begin{itemize}
\item a \emph{seam} of $P$ is an essential properly embedded arc in $P$ that has endpoints on distinct components of $\partial P$,
\item a \emph{wave} of $P$ is an essential properly embedded arc in $P$ that has endpoints on the same component of $\partial P$.
\end{itemize}
\end{defin}

\begin{figure}[ht]
\labellist
\footnotesize \hair 2pt
\pinlabel \textcolor{blue}{$w$} at 55 20
\pinlabel \textcolor{cyan}{$s$} at 23 20
\pinlabel $Q$ at 15 45
\endlabellist
\centering \scalebox{1.5}{\includegraphics{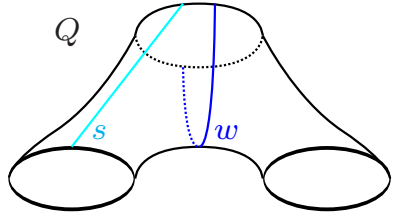}}
\caption{$w$ is a wave and $s$ is a seam of the pair of pants component $Q$.}\label{fig:waveseam}
\end{figure}

\begin{defin}
Let $a$ and $b$ be simple closed curves or arcs with fixed endpoints in $\Sigma$.  Then $a$ and $b$ admit a \emph{bigon} if there is an embedded disk $D$ in $\Sigma$ whose boundary is the endpoint union of a subarc of $a$ and a subarc of $b$.
\end{defin}

\begin{defin}
Two essential simple closed curves in $\Sigma$ (or two properly embedded arcs with fixed endpoints in an essential subsurface of $\Sigma$) \emph{intersect efficiently} if they do not admit a bigon.
\end{defin}

The condition that two curves $a$ and $b$ intersect efficiently is equivalent to the condition that the number of intersection points between $a$ and $b$ is minimal, up to (proper) isotopy (for a proof see \cite{FM}).  

\begin{defin}
Let $\mcP$ be a pants decomposition of $\Sigma$.  An embedded curve $\gamma$ is \emph{$k$-seamed with respect to $\mcP$} if it intersects $\mcP$ efficiently and contains at least $k$ copies of every isotopy class of seams of each pair of pants component of $\Sigma - \mcP$.
\end{defin}

The notion of a 1-seamed curve was introduced by Kobayashi \cite{Ko2}, who denoted such curves to be of \emph{full type}.  Note that if a curve $\gamma$ is $1$-seamed with respect to a pants decomposition $\mcP$ of $\Sigma$ then neither $\gamma$, nor any simple closed curve disjoint from $\gamma$, can contain a wave in any pants component of $\Sigma - \mcP$.

\begin{defin}
Let $\mcP$ be a pants decomposition of $\Sigma$.  A collection of essential simple closed curves $Y = \{y_1,y_2,...,y_t\}$ is \emph{collectively $k$-seamed with respect to $\mcP$} if $Y$ intersects $\mcP$ efficiently and $\bigcup\limits_{1 \leq i \leq t} y_i$ contains at least $k$ copies of every isotopy class of seams of each pair of pants component of $\Sigma - \mcP$.
\end{defin}

\begin{lem} \label{lem:compdiskwave}
Let $\mcD$ be a complete collection of compressing disks for $H$ and $\mcP = \partial \mcD \subset \partial H$.  If $\gamma$ is an essential simple closed curve in $\partial H$ that compresses in $H$ and intersects $\mcP$ efficiently, then $\gamma$ is either isotopic to a curve in $\mcP$ or $\gamma$ contains at least two waves of components of $\partial H - \mcP$.  
\end{lem}
\begin{proof}
Let $D \subset H$ be the disk that $\gamma$ bounds.  Isotope the interior of $D$ to intersect $\mcD$ minimally.  A standard innermost circle argument, exploiting the irreducibility of $H$, guarantees that each component of $D \cap \mcD$ is an arc.  If $D \cap \mcD = \emptyset$ then $\gamma$ is isotopic to a curve in $\mcP$. So suppose $D \cap \mcD \neq \emptyset$.

Let $\alpha$ be an outermost arc of $D \cap \mcD$ in $D$.  Then $\alpha$ and an arc $\beta \subset \gamma$ together bound a disk in $D$ that is disjoint from $\mcD$ in its interior.  As $\gamma$ intersects $\mcP$ efficiently, $\beta$ is an essential arc in a component of $\Sigma - \mcP$ and is therefore a wave.  Since $D - \mcD$ contains at least two distinct disks cut off by outermost arcs, $\gamma$ contains at least two distinct waves. \end{proof}

Note that Lemma \ref{lem:compdiskwave} implies that any meridian of $H$ intersects a $k$-seamed curve in at least $2k$ points.

\section{Dehn Twists and Twisting Number}

We first reproduce a definition of a standard Dehn twist as described in \cite{FM}.

\begin{defin}
Let $A$ be an annulus $S^1 \times [0,1]$ embedded in the $(\theta,r)$-plane by the map $(\theta,t) \rightarrow (\theta, t+1)$ and let the standard orientation of the plane induce an orientation on $A$.  Let $T: A \rightarrow A$ be the \emph{left twist map} of $A$ given by $T(\theta,t) = (\theta + 2 \pi t, t)$.  Similarly, let the \emph{right twist map} be given by $T^{-1}(\theta,t) = (\theta - 2 \pi t, t)$.
\end{defin}

\begin{defin}
Let $y$ be a simple closed curve in $\Sigma$ and $N$ an annular neighborhood of $y$.  Define $\phi$ to be an orientation-preserving homeomorphism from $A$ to $N$.  Then the \emph{Dehn twist operator along $y$} is the homeomorphism $\tau_y: \Sigma \rightarrow \Sigma$ given by:

\begin{displaymath}
   \tau_y = \left\{
     \begin{array}{ll}
       \phi \circ T \circ \phi^{-1}(p) & \mbox{if $p \in N$,}\\
       p & \mbox{if $p \in \Sigma - N$.}
     \end{array}
   \right.
\end{displaymath} 

$\tau_y$ is well-defined up to isotopy.  
\end{defin}

\begin{rmk}
In this paper, we have picked the convention that all positive powers of the Dehn twist operator will apply a left twist.
\end{rmk}

\begin{defin}
Suppose $Y = \{y_1,y_2,...,y_t\}$ is a collection of pairwise disjoint simple closed curves in $\Sigma$.  Let $\tau_Y: \Sigma \rightarrow \Sigma$ denote the composition of Dehn twists $\tau_{y_1} \circ \tau_{y_2} \circ ... \circ \tau_{y_n}$.  Since the $y_i$ are pairwise disjoint, it is easy to see that $\tau_Y$ is independent of the ordering of $\{y_i\}$.
\end{defin}

\begin{defin}
Suppose $N \subset \Sigma$ is an annulus with a specified $I$-fibration in an oriented surface $\Sigma$, and $c$ is a properly embedded arc in $N$ with fixed endpoints that intersects each $I$-fiber of $N$ efficiently.  Let $p$ be a point of intersection between $c$ and $\partial N$.  Then denote $v_{c}$ to be the inwards tangent vector of $c$ based at $p$.  Similarly, let $v_I$ be the inwards tangent vector based at $p$ of the $I$-fiber of $N$ that has $p$ as an endpoint.  Then $c$ \emph{turns left in $N$ at $p$} (resp. \emph{turns right in $N$ at $p$}) if the orientation determined by the pair $<v_I, v_{c}>$ (resp. the pair $<v_{c}, v_I>$) matches the corresponding orientation of the tangent space of $\Sigma$ at $p$.
\end{defin}

\begin{figure}[ht]
\labellist
\footnotesize \hair 2pt
\pinlabel \textcolor{green}{$c$} at 20 12
\pinlabel $N$ at -3 30
\pinlabel $v_{c}$ at 30 22
\pinlabel $v_I$ at 57 19
\pinlabel $p$ at 63 0
\endlabellist
\centering \scalebox{1.5}{\includegraphics{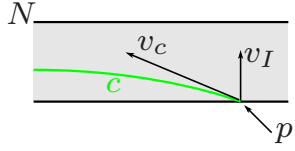}}
\caption{$c$ turns left in $N$.}\label{fig:turnleft}
\end{figure}

\begin{defin}
Let $N$ be an annulus and let $X = \{x_1,x_2,...,x_s\}$ be a collection of pairwise disjoint simple closed curves in $\Sigma$.  Then $N$ is \emph{fibered with respect to $X$} if $N$ has an $I$-fibration such that $N \cap X$ consists of $I$-fibers of $N$.
\end{defin}

The definition of the twisting number between two simple closed curves was first introduced by Lustig and Moriah in \cite{LM}.  To handle some additional subleties required for our arguments, we provide a modified definition.  First, we give the definition of the twisting number for a properly embedded essential arc with fixed endpoints in an annulus and later generalize to the twisting number between two simple closed curves.

\begin{defin}
Let $X = \{x_1,x_2,...,x_s\}$ be a collection of pairwise disjoint curves in $\Sigma$.  Suppose $N$ is an annulus in $\Sigma$ that is fibered with respect to $X$.  Let $c$ be an essential properly embedded arc in $N$ with fixed endpoints disjoint from $X \cap N$ that intersects each $I$-fiber of $N$ efficiently.  Set $m = \dfrac{|c \cap X|}{|N \cap X|}$.  Then $c$ has \emph{twisting number $m$} (resp. \emph{twisting number $(-m)$}) in $N$ with respect to $X$ if it turns left (resp. turns right) in $N$ at its endpoints.
\end{defin}

\begin{figure}[ht]
\labellist
\footnotesize \hair 2pt
\pinlabel {$B \subset \mathbb{R}^2$} at 0 190
\pinlabel {$A \subset \mathbb{R}^2$} at 0 100
\pinlabel $\tilde{T}$ at 112 145
\pinlabel $T$ at 112 56
\pinlabel -2 at 0 135
\pinlabel -1 at 25 135
\pinlabel 0 at 49 135
\pinlabel 1 at 73 135
\pinlabel 2 at 97 135
\pinlabel -1 at 45 116
\pinlabel 1 at 45 164
\pinlabel 2 at 45 188
\pinlabel -2 at 120 135
\pinlabel -1 at 145 135
\pinlabel 0 at 169 135
\pinlabel 1 at 193 135
\pinlabel 2 at 217 135
\pinlabel 1 at 165 164
\pinlabel 2 at 165 188
\pinlabel -1 at 165 116
\pinlabel -2 at 0 48
\pinlabel -1 at 25 48
\pinlabel 0 at 49 48
\pinlabel 1 at 73 48
\pinlabel 2 at 97 48
\pinlabel -2 at 48 0
\pinlabel -1 at 49 24
\pinlabel 1 at 49 72
\pinlabel 2 at 49 96
\pinlabel -2 at 120 48
\pinlabel -1 at 145 48
\pinlabel 0 at 169 48
\pinlabel 1 at 193 48
\pinlabel 2 at 217 48
\pinlabel -2 at 168 0
\pinlabel -1 at 169 24
\pinlabel 1 at 169 72
\pinlabel 2 at 169 96
\endlabellist
\centering \scalebox{1.5}{\includegraphics{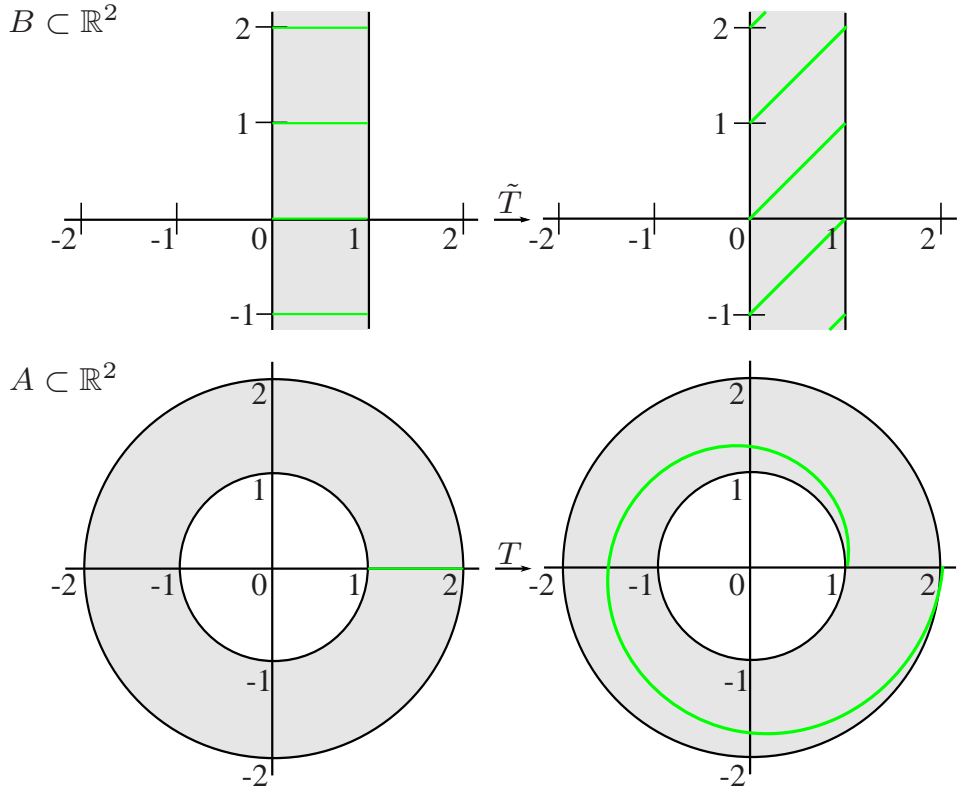}}
\caption{The left twist map $T$ on the annulus $A$ and the induced map $\tilde{T}$ on $B = [0,1] \times \mathbb{R}$.}\label{fig:BtoA}
\end{figure}

In the following lemmas, it is helpful to work in the universal cover of the annulus.  Recall that $A$ is an annulus $S^1 \times [0,1]$ embedded in the $(\theta,r)$-plane by the map $(\theta,t) \rightarrow (\theta,t+1)$.  Then $B = [0,1] \times \mathbb{R} \subset \mathbb{R}^2$ is the universal cover of $A$ with covering map $\psi:B \rightarrow A$ given by $\psi(x,y) = (2 \pi y,x+1)$.  The left twist map $T: A \rightarrow A$ induces a homeomorphism $\tilde{T}:B \rightarrow B$ given by $\tilde{T}(x,y) = (x,y+x)$ (see Figure \ref{fig:BtoA}) so that the following diagram commutes:
\begin{equation*} \begin{CD}
B @>{\tilde{T}}>> B\\
@VV{\psi}V @VV{\psi}V\\
A @>{T}>> A
\end{CD}
\end{equation*}

\vspace{5pt}

If $N$ is an embedded annulus in $\Sigma$, we can compose $\psi$ with an orientation-preserving homeomorphism $\phi: A \rightarrow N$ to consider $B$ as the universal cover of $N$.  Choose $\phi$ so that the $I$-fibers of $N$ lift to horizontal arcs in $B$.  A properly embedded arc in $N$ that intersects the $I$-fibers of $N$ efficiently can be isotoped rel its endpoints so that it lifts to a collection of straight line segments in $B$, each with the same slope, that differ by a vertical translation of an integral distance.  Applying the induced left twist map $\tilde{T}$ (resp. right twist map) to a straight line segment in $B$ then increases (resp. decreases) the slope by 1.  

Now let $X = \{x_1,x_2,...,x_s\}$ be a collection of pairwise disjoint simple closed curves on $\Sigma$ that intersect $N$ in $I$-fibers. Suppose $n = |N \cap X|$ and then, after possibly rechoosing $\phi$, the collection of all lifts of the arcs of $N \cap X$ are exactly the horizontal arcs in $B$ of the form $[0,1] \times \frac{z}{n}$ for all $z \in \mathbb{Z}$.  Then a properly embedded straight line segment $\alpha$ in $B$ with endpoints disjoint from lifts of $N \cap X$ and with slope $m$ projects to a properly embedded arc in $N$ with twisting number $\frac{a}{n}$ with respect to $X$, where $a \in \mathbb{Z}$ is chosen so that $|a|$ is the total number of intersections between $\alpha$ and lifts of $N \cap X$, and $a$ and $m$ share the same sign (see Figure \ref{fig:cover}).

\begin{figure}[ht]
\labellist
\footnotesize \hair 2pt
\pinlabel \textcolor{green}{$c$} at 47 87
\pinlabel \textcolor{red}{$X$} at 93 95
\pinlabel $N$ at 0 24
\pinlabel \textcolor{red}{$X$} at 187 12
\pinlabel \textcolor{green}{$\tilde{c}$} at 158 64
\pinlabel {0} at 140 16
\pinlabel {$1$} at 140 48
\pinlabel {$2$} at 140 80
\pinlabel {0} at 147 0
\pinlabel {1} at 172 0
\pinlabel {$B = [0,1] \times \mathbb{R}$} at 160 98
\endlabellist
\centering \scalebox{1.5}{\includegraphics{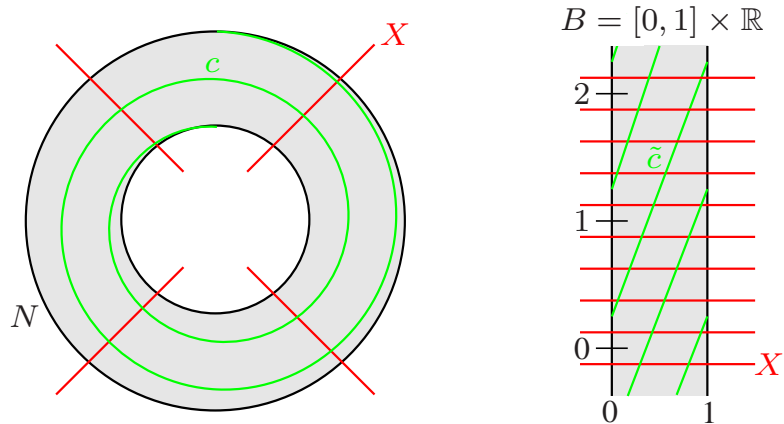}}
\caption{An arc $c$ with twisting number 2 with respect to $X$ lifts to arcs with slope 2 in $B = [0,1] \times \mathbb{R}$.}\label{fig:cover}
\end{figure}

With this perspective we can make the following observations.

\begin{lem} \label{lem:disjointcircling}
Let $X = \{x_1,x_2,...,x_s\}$ be a collection of pairwise disjoint curves in $\Sigma$.  Suppose $N$ is an annulus in $\Sigma$ that is fibered with respect to $X$.  Let $c_1$ and $c_2$ be disjoint properly embedded arcs in $N$ with twisting numbers $m_1$ and $m_2$ in $N$ respectively with respect to $X$.  Then $|m_1 - m_2| \leq 1$.
\end{lem}
\begin{proof}
Let $n = |N \cap X|$ and let $B = [0,1] \times \mathbb{R}$ denote the universal cover of $N$ with covering map as defined previously so that the lifts of the arcs $N \cap X$ is the set $\{[0,1] \times \frac{z}{n} | z \in \mathbb{Z}\}$.  Then $c_1$ and $c_2$ can be isotoped so that the collection of all of their lifts, $\tilde{c_1}$ and $\tilde{c_2}$ respectively, are a collection of linear arcs in $B$ with slopes equal to their twisting numbers, i.e. $m_1$ and $m_2$.  Since $c_2$ is disjoint from $c_1$, any lift of $c_2$ must lie in a parallelogram-shaped region $R$ of $B - \tilde{c_1}$.  The vertical distance between any consecutive lifts of $c_1$ in $B$ is equal to 1, so the slope $m_2$ of $c_2$ must satisfy $m_1 - 1 \leq m_2 \leq m_1 + 1$. 
\end{proof}

\begin{lem} \label{lem:LMtwist}
Let $X = \{x_1,x_2,...,x_s\}$ be a collection of pairwise disjoint curves in $\Sigma$.  Suppose $y$ is a simple closed curve in $\Sigma$ and $N$ is an annular neighborhood of $y$ that is fibered with respect to $X$.  Let $c$ be a properly embedded arc in $N$ with twisting number $m$ with respect to $X$.  Then $\tau^k_y(c)$ has twisting number $m+k$ in $N$ with respect to $X$.
\end{lem}
\begin{proof}
Again let $n = |N \cap X|$ and let $B = [0,1] \times \mathbb{R}$ be the universal cover of $N$ such that the set of all lifts of $N \cap X$ is equal to $\{[0,1] \times \frac{z}{n} | z \in \mathbb{Z}\}$.  Then $c$ can be isotoped so that the collection of all of its lifts in $B$, denoted as $\tilde{c}$, is a collection of linear arcs in $B$ with slope $m$.  Applying $\tilde{T}^k$ to $B$ sends each linear arc of $\tilde{c}$ with slope $m$ to a linear arc with slope $m+k$.  Projecting a component of $\tilde{T}^k(\tilde{c})$ back down to $N$ yields a properly embedded arc in $N$ that agrees with $\tau^k_y(c)$ and has twisting number $m+k$ with respect to $X$.
\end{proof}

The next goal is to define the twisting number for a simple closed curve $\gamma$ about a simple closed curve $y$ that $\gamma$ intersects efficiently.  This is more complicated than defining the twisting number for an arc in an annulus; given an annular neighborhood $N$ of $y$, ambient isotopies of $\gamma$ can modify the twisting number of any component of $\gamma \cap N$.

Let $X = \{x_1,x_2,...,x_s\}$ and $Y = \{y_1,y_2,...,y_t\}$ be collections of pairwise disjoint curves such that the components of $X$ and $Y$ intersect efficiently, and suppose $N = \{N_1,N_2,...,N_t\}$ is a collection of pairwise disjoint annuli such that each $N_i$ is an annular neighborhood of $y_i$ and $N_i$ is fibered with respect to $X$.  

\begin{defin}
If $\gamma$ is a simple closed curve that intersects $X$, $Y$, $\partial N$, and each $I$-fiber of $N$ efficiently, then we will say $\gamma$ is in \emph{efficient position} with respect to $(X,Y,N)$.
\end{defin}

We will also require the use of the following helpful fact (see \cite{HS} for a proof):

\begin{lem} \label{lem:efficientisotopy}
Let $Y_1,...,Y_{n-1}$ each be a collection of disjoint essential simple closed curves in $\Sigma$ so that each pair $Y_i$, $Y_j$ intersects efficiently.  If $Y_n$ is another collection of disjoint essential simple closed curves, then $Y_n$ can be isotoped to intersect each of the other sets efficiently without disturbing the efficient intersection of the others.  Moreover, as long as no component of $Y_n$ is parallel to a component of any $Y_i$, any two embeddings of $Y_n$ that intersect $Y_1,...,Y_{n-1}$ efficiently are isotopic through an isotopy which keeps all intersections efficient (though during the isotopy $Y_n$ may pass over intersection points of $Y_i$ with $Y_j$).
\end{lem}

For our purposes, Lemma \ref{lem:efficientisotopy} implies that as long as $\gamma$ is not parallel to a component of $X$ or $Y$, then between any two embeddings of $\gamma$ that are in efficient position with respect to $(X,Y,N)$, there exists an isotopy $f_t$ (with $t \in [0,1])$ such that $f_t(\gamma)$ intersects $X$, $Y$, and $\partial N$ efficiently for all $t$.  Moreover, after perturbing $f_t$ in $N$, we can assume $f_t(\gamma)$ is in efficient position with respect to $(X,Y,N)$ for all $t$.  Then $f_t$ induces an isotopy $f^N_t$ of $\Sigma$ that is transverse to $\partial N \cup Y$ and that can restrict to an isotopy on $N$.  Let $c_1,c_2,...,c_r$ denote the components of $\gamma \cap N$ with twisting numbers $m_1,m_2,...,m_r$ respectively in $N$.  We can associate $c_1,c_2,...,c_r$ with their images under $f^N_t$ and then track the corresponding changes in their twisting numbers.  Using such isotopies, we want to isotope $\gamma$ to maximize $\sum_j |m_j|$.

\begin{defin}
Let $\gamma$ be a simple closed curve that is in efficient position with respect to $(X,Y,N)$.  Suppose there exists a triangle $E$ in $\Sigma$ with a side $s_X$ in $X$, a side $s_N$ in $\partial N$, and a side $s_{\gamma}$ in $\gamma$, such that $\mathring E$ is disjoint from $N$. Then the triangle $E$ will be called an \emph{outer triangle of $N$} (see Figure \ref{fig:triangle}).
\end{defin}

\begin{figure}[ht]
\labellist
\footnotesize \hair 2pt
\pinlabel \textcolor{green}{$s_{\gamma}$} at 42 25
\pinlabel \textcolor{red}{$s_X$} at 6 23
\pinlabel \textcolor{blue}{$c$} at 76 0
\pinlabel $s_{N}$ at 35 4
\pinlabel $N$ at -5 0
\pinlabel $E$ at 27 17
\endlabellist
\centering \scalebox{1.5}{\includegraphics{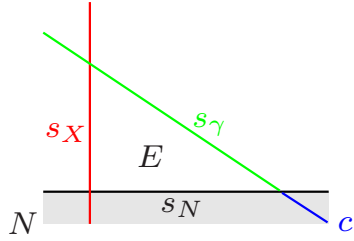}}
\caption{A outer triangle of $N$.}\label{fig:triangle}
\end{figure}

Given an outer triangle $E$ of $N$, we can perform an isotopy supported in an open neighborhood of $E$ of $\gamma$ that pushes $s_{\gamma}$ and all other arcs of $\mathring E \cap \gamma$ into $N$ (see Figure \ref{fig:properties1-5}) such that $\gamma$ will be in efficient position with respect to $(X,Y,N)$ throughout the isotopy. Note that this isotopy increases $\sum_j |m_j|$.

\begin{figure}[ht]
\labellist
\footnotesize \hair 2pt
\pinlabel \textcolor{green}{$\gamma$} at -3 40
\pinlabel \textcolor{blue}{$c_j$} at 72 19
\pinlabel \textcolor{red}{$X$} at 17 48
\pinlabel $N$ at -5 7
\pinlabel $E$ at 25 31
\pinlabel \textcolor{green}{$\gamma$} at 92 40
\pinlabel \textcolor{blue}{$c_j$} at 167 19
\pinlabel \textcolor{red}{$X$} at 113 48
\pinlabel $N$ at 91 7
\endlabellist
\centering \scalebox{1.5}{\includegraphics{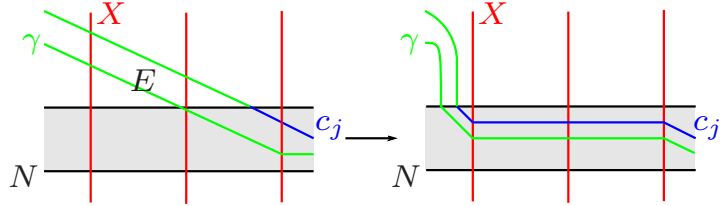}}
\caption{An isotopy of $\gamma$ across an outer triangle of $N$ that increases $|m_j|$.}\label{fig:properties1-5}
\end{figure}

As $|\gamma \cap X|$ is finite, a finite number of isotopies of $\gamma$ across outer triangles of $N$ will yield an embedding of $\gamma$ in efficient position with respect to $(X,Y,N)$ that admits no such outer triangles.  We claim that for such embeddings, $\sum_j |m_j|$ is maximal over all representatives of the isotopy class of $\gamma$ that are in efficient position with respect to $(X,Y,N)$.

\begin{prop}
Suppose $\gamma$ is a simple closed curve in efficient position with respect to $(X,Y,N)$ and $\gamma$ is not parallel to a component of $Y$.  Let $c_1,c_2,...,c_r$ denote the components of $\gamma \cap N$ and let $m_j$ be the twisting number of $c_j$.  Then $\sum_j |m_j|$ is maximal over all embeddings of $\gamma$ that are in efficient position with respect to $(X,Y,N)$ if and only if $\gamma$ admits no outer triangles of $N$.
\end{prop}
\begin{proof}
Note that in the case that $\gamma$ is parallel to a component of $X$, then $\gamma$ admits no outer triangles (since it is disjoint from $X$) and $\sum_j |m_j| = 0$ is maximal.

If $\gamma$ is an embedding in efficient position with respect to $(X,Y,N)$ such that $\sum_j |m_j|$ is maximal, then $\gamma$ admits no outer triangles, as otherwise an isotopy across such a triangle would increase $\sum_j |m_j|$.

Conversely, suppose that $\gamma$ admits no outer triangles of $N$, but $\sum_j |m_j|$ is not maximal.  Then there exists an isotopy $f_t$ of $\gamma$ that increases $\sum_j |m_j|$.  Moreover, by Lemma \ref{lem:efficientisotopy} we can assume $f_t(\gamma)$ intersects $X$, $Y$, and $\partial N$ efficiently for all $t \in [0,1]$.  This implies $f_t$ induces isotopies $f_t^X$ and $f_t^N$ which are transverse to $X$ and $Y \cup \partial N$ respectively such that $f_t^X(\gamma) = f_t^N(\gamma) = f_t(\gamma)$ for all $t$. 

Since $f_t$ increases $\sum_j |m_j|$, there exists a point $p \in \gamma \cap X$ such that $p$ lies off of $N$, but $f_1^X(p)$ lies in $N$.  Let $t_*$ denote the smallest value such that $f_{t_*}^X(p) \in \partial N$.  As $f_{t_*}^X(p)$ lies in $\partial N$, $f_{t_*}^X(p) = f_{t_*}^N(q)$ for some point $q \in \gamma \cap \partial N$.

Since we chose $t_*$ to be the smallest value such that $f_{t_*}^X(p) \in \partial N$, there exists a subarc $s$ of $\gamma$ connecting $p$ to $q$ such that $\mathring s$ is disjoint from $N$.  Then as $f_t^X$ pushes $p$ along $X$ across $\partial N$ and $f_t^N$ pushes $q$ along $\partial N$ across $X$, $s$ must form the side of an outer triangle $E$ of $N$, where the other two sides of $E$ consist of a subarc of $X$ connecting $p$ to $f_{t_*}^X(p)$ and a subarc of $\partial N$ connecting $q$ to $f_{t_*}^N(q)$.  This contradicts our assumption that $\gamma$ admitted no such outer triangles of $N$, and hence $\sum_j |m_j|$ is maximal.
\end{proof}

We can now provide the definition of twisting number about a simple closed curve.

\begin{defin}
Let $y$ be a simple closed curve that intersects $X$ efficiently and $N$ be an annular neighborhood of $y$ that is fibered with respect to $X$.  Suppose $\gamma$ is an essential simple closed curve that intersects $y$ non-trivially.  Then define the \emph{twisting number of $\gamma$ in $N$}, denoted as $tw(\gamma,N)$, to be the maximum twisting number of a component of $\gamma \cap N$ over all representatives of the isotopy class of $\gamma$ that are in efficient position with respect to $(X,y,N)$ and admit no outer triangles of $N$.  
\end{defin}

We require $\gamma$ to intersect $X$ efficiently as otherwise there would be no maximum to the twisting number of components of $\gamma \cap N$.  We include the assumption that $\gamma$ does not admit outer triangles of $N$ as otherwise we could push any ``negative twisting'' of $\gamma$ about $y$ to lie outside of $N$ and therefore $tw(\gamma,N)$ would always be non-negative.

Suppose that $N'$ is another choice of annular neighborhood of $y$ that is fibered with respect to $X$.  Then there exists an isotopy $g_t$ that sends $N$ to $N'$ such that $I$-fibers of $N$ are sent to $I$-fibers of $N'$ and $X$ is fixed.  In particular, this means for any component $c$ of $\gamma \cap N$, $g_1(c)$ is a component of $\gamma \cap N'$.  Moreover, if $\gamma$ is in efficient position with respect to $(X,y,N)$ and does not admit any outer triangles of $N$, then $g_1(\gamma)$ is in efficient position with respect to $(X,y,N')$ and does not admit any outer triangles of $N'$.  Hence we can conclude that $tw(\gamma,N) = tw(\gamma,N')$ and twisting number is independent of the choice of annular neighborhood of $y$.  We will thus denote the twisting number as $tw(\gamma,y)$. 

Given our definition of twisting number and Lemma \ref{lem:LMtwist}, one would expect for there to be a natural relationship between $tw(\gamma,y)$ and $tw(\tau_y(\gamma),y)$.  If an embedding of $\gamma$ admits no outer triangles of $N$, by the following lemma we can then obtain an embedding of $\tau^k_y(\gamma)$ via arc replacements of the components of $\gamma \cap N$ that, for $k$ sufficiently large, continues to intersect $X$ efficiently.  This result and its corollary are integral to proving Lemma \ref{lem:pantswinding}, which provides the foundation for our later results.

\begin{lem} \label{lem:winding}
We continue to define $X$, $Y$, and $N$ as before ($Y$ is allowed to have multiple components).  Suppose $\gamma$ is a simple closed curve in efficient position with respect to $(X,Y,N)$ and admits no outer triangles of $N$.  Denote the components of $\gamma \cap N$ as $c_1,c_2,...,c_r$ and let $m_j$ be the twisting number of $c_j$ in the component of $N$ that contains $c_j$. Let $k \in \mathbb{Z}$ such that either for all $j$, $k + m_j \geq 0$ or for all $j$, $k + m_j \leq 0$.  Then there exists an embedding of $\tau^k_Y(\gamma)$ in efficient position with respect to $(X,Y,N)$ that admits no outer triangles of $N$ and the components of $\tau^k_Y(\gamma) \cap N$ consist of arcs $\{c'_1,c'_2,...,c'_r\}$ where $c'_j$ has twisting number $m_j + k$ for each $j$.
\end{lem}
\begin{proof}
Let $\gamma' = \tau_Y^k(\gamma)$.  By Lemma \ref{lem:LMtwist}, we can obtain an embedding of $\gamma'$ by fixing $\gamma$ off of $N$ and replacing each component $c_j$ of $\gamma \cap N$ with a properly embedded arc $c'_j$ in $N$ that has twisting number $m_j+k$.  We replace these arcs carefully so that this embedding of $\gamma'$ intersects $Y$ and each $I$-fiber of $N$ efficiently.  As $\gamma$ intersects $\partial N$ efficiently, this embedding of $\gamma'$ will as well.  So it remains to show that this embedding of $\gamma'$ intersects $X$ efficiently.

Suppose otherwise.  Then there exists a component $b'$ of $\gamma' - X$ that cobounds a bigon $D$ with a subarc of some $x_{\ell} \in X$.  As $\gamma$ intersects $X$ efficiently, $b'$ cannot be contained in the complement of $N$ since it would then coincide with $\gamma$.  Moreover, as the components of $\gamma' \cap N$ intersects each $I$-fiber of $N$ efficiently, $b'$ cannot be contained in $N$.  So $b'$ and $\partial N$ must intersect at least once and we can consider the two distinct components $\beta_1$ and $\beta_2$ of $b' - \partial N$ which connect $x_{\ell}$ with $\partial N$ and have interiors disjoint from $\partial N$.

Suppose the interior of $\beta_1$ lies off of $N$.  Then $\beta_1$ forms a side of an outer triangle of $N$.  However, as this embedding of $\gamma'$ agrees with $\gamma$ off of $N$, $\gamma'$ also does not admit any outer triangles of $N$.  Hence we have a contradiction and $\beta_1$ must instead be contained in $N$.  Repeating the argument with $\beta_2$ implies $\beta_2$ must also lie in $N$.

So both $\beta_1$ and $\beta_2$ are subarcs of components of $\gamma' \cap N$.  By hypothesis, $k$ was chosen so every component of $\gamma' \cap N$ has twisting number with the same sign.  On the other hand, the arcs of $\partial N \cap D$ must have one endpoint on $x_{\ell}$ and one endpoint on $b'$, as otherwise $\partial N$ would not intersect $x_{\ell}$ or $b'$ efficiently.  This means that $\beta_1$ and $\beta_2$ have twisting numbers of opposite sign in $N$ (see Figure \ref{fig:nobigon}), a contradiction.

Hence $\gamma'$ must have efficient intersection with $X$.
\end{proof}

\begin{figure}[ht]
\labellist
\footnotesize \hair 2pt
\pinlabel \textcolor{green}{$\beta_1$} at 0 10
\pinlabel \textcolor{green}{$\beta_2$} at 73 10
\pinlabel \textcolor{red}{$x_{\ell}$} at 37 0
\pinlabel $N$ at 10 45
\pinlabel $N$ at 62 45
\pinlabel $D$ at 37 20
\pinlabel \textcolor{green}{$b'$} at 37 49
\endlabellist
\centering \scalebox{1.5}{\includegraphics{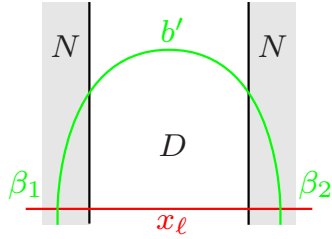}}
\caption{In this case, $\beta_1$ has positive twisting number in $N$ and $\beta_2$ has negative twisting number in $N$.}\label{fig:nobigon}
\end{figure}

In particular, Lemma \ref{lem:winding} provides a lower bound for the twisting number of $\tau^k_y(\gamma)$ about $y$ with respect to $X$ for appropriately chosen $k$.

\begin{cor} \label{cor:winding}
Let $X$, $Y$, $N$, and $\gamma$ be defined as in Lemma \ref{lem:winding}.  Let $m_1,m_2,...,m_r$ be the twisting numbers of the components of $\gamma \cap N$ and choose $k \in \mathbb{Z}$ so that for all $j$, $k + m_j \geq 0$.  If a component of $\gamma \cap N_{\ell}$ has twisting number equal to $m$, for some $N_{\ell} \in N$, then $tw(\tau^k_Y(\gamma),y_{\ell}) \geq m + k$.
\end{cor}
\begin{proof}
Suppose $c$ is the component of $\gamma \cap N_{\ell}$ with twisting number equal to $m$ in $N_{\ell}$.  Then by Lemma \ref{lem:winding}, we obtain an embedding of $\tau^k_Y(\gamma)$ that is in efficient position with respect to $(X,Y,N)$ by replacing $c$ with an arc $c'$ in $N_{\ell}$ that has twisting number equal to $m+k$.  Hence $m+k$ is a lower bound for $tw(\tau^k_Y(\gamma),y_{\ell})$.
\end{proof}

We now replace the arbitrary collection of pairwise disjoint curves $X$ in $\Sigma$ with a pants decomposition $\mcP$ of $\Sigma$.  The twisting number of $\gamma$ around a 1-seamed curve $y$ then corresponds to the number of seams $\gamma$ contains and yields the following natural relationship.

\begin{rmk}
Let $\mcP$ be a pants decomposition of $\Sigma$ and $y$ a simple closed curve that is 1-seamed with respect to $\mcP$.  If $\gamma$ is a simple closed curve such that $|tw(\gamma,y)| > k$, then $\gamma$ is $k$-seamed with respect to $\mcP$.  
\end{rmk}

When applied in this setting, Lemma \ref{lem:winding} gives the following result.  Note that for part (c), we show that adding the assumption that $Y$ is collectively 2-seamed with respect to $\mcP$ can strengthen the results shown in parts (a) and (b). Since each component of $Y$ is assumed to be 1-seamed, $Y$ is always collectively 2-seamed if $Y$ has at least 2 components.

\begin{lem} \label{lem:pantswinding}

Let $H$ be a handlebody with $\partial H = \Sigma$ and $\mcP$ a pants decomposition for $H$.  Let $Y = \{y_1,y_2,...,y_t\}$ be a collection of pairwise disjoint curves in $\Sigma$ such that $\mcP$ and $Y$ intersect efficiently and $y_i$ is 1-seamed with respect to $\mcP$ for each $i$.  Suppose $\gamma$ is a simple closed curve that intersects $\mcP$ and $Y$ efficiently.

(a) If $\gamma$ is not 1-seamed with respect to $\mcP$ and $\gamma \cap y_{\ell} \neq \emptyset$ for some $y_{\ell}$, then for $k \geq 1$,
$$tw(\tau_Y^k(\gamma),y_{\ell}) \geq k-1.$$

(b) If $\gamma$ bounds a disk in $H$, then for every $y_i \in Y$ and $k \geq 1$,
$$tw(\tau_Y^k(\gamma),y_i) \geq k.$$

(c) If $Y$ is collectively 2-seamed and $k \geq 2$ then the inequalities proved in parts (a) and (b) become strict inequalities.
\end{lem}
\begin{proof}
Let $N = \{N_1,N_2,...,N_t\}$ be a collection of pairwise disjoint annuli such that $N_i$ contains $y_i$ and $N_i$ is fibered with respect to $\mcP$.  Isotope $\gamma$ to be in efficient position with respect to $(\mcP,Y,N)$ and so it admits no outer triangles of $N$.  Let $c_1,c_2,...,c_r$ be the components of $\gamma \cap N$ with twisting numbers $m_1,m_2,...,m_r$ respectively.

\vspace{10pt}

\noindent \textit{Proof of part (a).}
As mentioned in the earlier remark, if $|m_j| > 1$ for any $j$ then $\gamma$ would be 1-seamed with respect to $\mcP$, since every component of $Y$ is 1-seamed.  Since this would contradict our hypothesis, $|m_j| \leq 1$.

Then for $k \geq 1$, Corollary \ref{cor:winding} implies $tw(\tau^k_Y(\gamma),y_{\ell})$ is bounded below by $k$ plus the twisting number of any component of $\gamma \cap N_{\ell}$.  As these twisting numbers are all at least -1, we get $tw(\tau^k_Y(\gamma),y_{\ell}) \geq k - 1$.

\vspace{10pt}

\noindent \textit{Proof of part (b).}
By Lemma \ref{lem:compdiskwave}, $\gamma$ is either parallel to a component of $\mcP$ or contains a wave of a component of $\Sigma - \mcP$.

If $\gamma$ is parallel to a component of $\mcP$, then for every $y_i \in Y$, $y_i$ intersects $\gamma$ and $\gamma \cap N_i$ consists of arcs with twisting number 0.  Therefore, for $k \geq 1$, Corollary \ref{cor:winding} implies $tw(\tau_Y^k(\gamma),y_i) \geq k$.

So suppose $\gamma$ instead contains a wave $w$ of a component of $\Sigma - \mcP$.  Then since every 1-seamed curve must intersect $w$, each $y_i$ intersects $w$ and moreover $w \cap N_i$ consists of properly embedded arcs in $N_i$ with twisting number $0$.  This implies $m_1,m_2,...m_r$ are all greater than or equal to $-1$ by Lemma \ref{lem:disjointcircling}.  So, for $k \geq 1$, Corollary \ref{cor:winding} implies $tw(\tau_Y^k(\gamma),y_i) \geq k$.

\vspace{10pt}

\noindent \textit{Proof of part (c).}
What we prove here is actually a slightly more general claim.

\vspace{10pt}

\textit{Claim:} Define $\mcP$, $Y$, and $N$ as before and also assume $Y$ is collectively 2-seamed. Let $\gamma'$ be a simple closed curve that is in efficient position with respect to $(\mcP,Y,N)$ such that every component of $\gamma' \cap N$ has strictly positive twisting number.  Then if a component $c'$ of $\gamma' \cap N$ has twisting number equal to a positive integer, then there exists an isotopy of $\gamma'$ that keeps $\gamma'$ in efficient position with respect to $(\mcP,Y,N)$ and increases the twisting number of $c'$.

\vspace{10pt}

\textit{Proof of claim.} 
Since $Y$ is collectively 2-seamed, one of the endpoints of $c'$ lies on the boundary of some rectangular component $R$ of $\Sigma - (N \cup \mcP)$.  Denote this endpoint as $r$ and assume that we have perturbed $\gamma'$ so that $r$ lies off of $\mcP$. 

Consider the subarc $\beta$ of $\gamma' \cap R$ that has $r$ as one endpoint.  Then the second endpoint of $\beta$ lies on one of the four sides of $R$.  However, $\beta$ cannot have both endpoints on the same side of $R$, as this would create a bigon between $\gamma'$ and $\partial N$.  Moreover, if we follow along $\beta$ starting at $r$, we see that $\beta$ cannot turn right in $R$ and have its second endpoint on $\mcP$, as then $\gamma'$ would form a bigon with $\mcP$.  Thus, $\beta$ must either turn to the left with its second endpoint on $\mcP$, or have its second endpoint on the opposite side of $R$ on $\partial N$.  In the latter case, as all components of $\gamma' \cap N$ have strictly positive twisting number, $\gamma'$ must subsequently turn left in $N$.  Note that since $c'$ has twisting number equal to an integer, the second endpoint of $\beta$ cannot also be an endpoint of $c'$.  

Therefore, as shown in Figure \ref{fig:2seamedrect}, in either case there exists an isotopy of $\gamma'$ that preserves the efficient position of $\gamma'$ with respect to $(\mcP,Y,N)$ and increases the twisting number of $c'$.  This concludes the proof of the claim.

\vspace{10pt}

From part (a), we obtain the inequality that $tw(\tau^k_Y(\gamma),y_{\ell}) \geq k-1$ for some $y_{\ell} \in Y$.  Assuming $k \geq 2$ and $Y$ is collectively 2-seamed, the above claim shows that any embedding of $\tau^k_Y(\gamma)$ such that a component of $\tau^k_Y(\gamma) \cap N_{\ell}$ has twisting number equal to $k-1$ can be isotoped so that this twisting number is increased.  Hence $tw(\tau^k_Y(\gamma),y_{\ell}) > k-1$.  

An identical argument shows that, with the added assumptions of part (c), we also achieve strict inequality in part (b).
\end{proof}

\begin{figure}[ht]
\labellist
\footnotesize \hair 2pt
\pinlabel \textcolor{green}{$\beta$} at 36 37
\pinlabel $r$ at 48 36
\pinlabel \textcolor{blue}{$c'$} at 68 17
\pinlabel \textcolor{blue}{$c'$} at 152 17
\pinlabel $N$ at -5 8
\pinlabel $N$ at 79 8
\pinlabel \textcolor{red}{$\mathcal{P}$} at 14 69
\pinlabel \textcolor{red}{$\mathcal{P}$} at 50 69
\pinlabel \textcolor{red}{$\mathcal{P}$} at 98 69
\pinlabel \textcolor{red}{$\mathcal{P}$} at 134 69
\endlabellist
\centering \scalebox{1.5}{\includegraphics{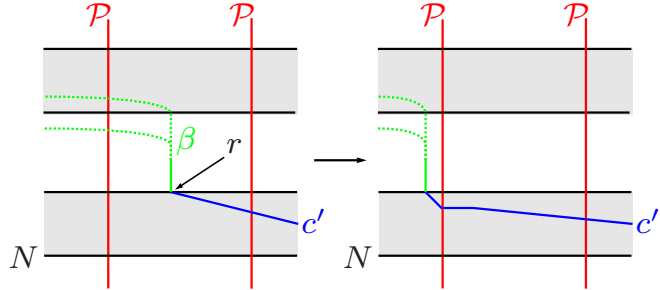}}
\caption{$\beta$ is isotoped into $N$, increasing the twisting number of $c'$.}\label{fig:2seamedrect}
\end{figure}

\section{Heegaard Splittings and Hempel Distance} \label{sec:heegaard}

\begin{defin}
A collection of pairwise disjoint essential simple closed curves $Z$ in $\Sigma$ is a \emph{full set of curves} if $\Sigma - Z$ is a collection of punctured spheres.
\end{defin}
In particular, both a standard cut system and a pants decomposition of $\Sigma$ are full.

Suppose $Z$ is a full set of curves in $\Sigma$.  Consider the 3-manifold obtained from $\Sigma \times I$ by attaching 2-handles to $\Sigma \times \{1\}$, one along each curve in $Z$.  Since $\Sigma - Z$ consists of a collection of punctured spheres, the boundary of the 3-manifold consists of $\Sigma \times \{0\}$ together with a collection of spheres. Attaching a 3-ball to each of the spheres creates a handlebody whose boundary is $\Sigma$.  Denote the handlebody by $V_Z$, with $\partial V_Z = \Sigma$.

\begin{defin}
A genus $g$ \emph{Heegaard splitting} of a closed 3-manifold $M$ is a decomposition of $M$ into the union of two genus $g$ handlebodies, identified along their boundaries.
\end{defin}

\begin{rmk}
Let $Z$ and $Z'$ be two full sets of curves in $\Sigma$.  Then $V_{Z}$ and $V_{Z'}$ are handlebodies and $M = V_{Z} \cup_{\Sigma} V_{Z'}$ is a Heegaard splitting of the closed orientable manifold $M$.
\end{rmk}

Given two full sets of curves $Z$ and $Z'$ in $\Sigma$ such that $\Sigma - Z$ and $\Sigma - Z'$ are homeomorphic (for example, if both $Z$ and $Z'$ are standard cut systems), there exists a surface automorphism of $\Sigma$ that sends $Z$ to $Z'$ (for more details see \cite{FM}).  This is the map that identifies the boundary of $V_Z$ to $V_{Z'}$.  Conversely, if we start with just $Z$ and $h$ is a surface automorphism of $\Sigma$, then we can obtain a 3-manifold with Heegaard splitting $V_Z \cup_{\Sigma} V_{h(Z)}$.  In this way, every surface automorphism of $\Sigma$ induces a Heegaard splitting of some closed 3-manifold.

\begin{defin}
The \emph{curve complex} for $\Sigma$, denoted as $C(\Sigma)$, is the complex whose vertices are isotopy classes of essential simple closed curves in $\Sigma$ and a set of distinct vertices $\{v_0,v_1,...,v_k\}$ determines a $k$-simplex if they are pairwise disjoint.  
\end{defin}
Note that in this paper we only need to consider the 1-skeleton of the curve complex.

\begin{defin}
Given two collections of simple closed curves $\mathcal{Z}$ and $\mathcal{Z}'$ in $\Sigma$, the \emph{distance between $\mathcal{Z}$ and $\mathcal{Z}'$}, denoted $d(\mathcal{Z},\mathcal{Z}')$, is the minimal number of edges in a path in $C(\Sigma)$ between a vertex in $\mathcal{Z}$ and a vertex in $\mathcal{Z}'$.
\end{defin}

\begin{defin}
Let $H$ be a handlebody with $\partial H = \Sigma$.  Then the \emph{disk complex of $H$}, denoted as $D(H)$, is the subcomplex of $C(\Sigma)$ determined by all essential simple closed curves that bound embedded disks in $H$.  If $Z$ is a full set of curves in $\Sigma$, we will let $K_Z$ denote $D(V_Z)$.
\end{defin}

The following definition was first introduced by Hempel \cite{He}.

\begin{defin}
Let $Z$ and $Z'$ be full sets of curves in $\Sigma$. Then the \emph{Hempel distance}, or \emph{distance}, of the Heegaard splitting $V_Z \cup V_{Z'}$ is equal to $d(K_Z,K_{Z'})$.
\end{defin}

We now introduce the useful term ``diskbusting.''

\begin{defin}
Let $H$ be a handlebody with $\partial H = \Sigma$.  A collection of pairwise disjoint essential curves $Y = \{y_1,y_2,...,y_t\}$ is \emph{diskbusting} on $H$ if every meridian of $H$ intersects some $y_i \in Y$.
\end{defin}

Observe that a simple closed curve $y$ is diskbusting on $H$ if and only if $d(D(H),y) \geq 2$.  On the other hand, there exist collections of curves $Y$ such that $Y$ is diskbusting on $H$ and $d(D(H),y_i) = 1$ for all $y_i \in Y$. For an example, see the collection $Y_2$ in Figure \ref{fig:diskbusting}.

\begin{defin}
Let $X$ be a standard cut system and $Y = \{y_1,y_2,...,y_t\}$ a collection of essential curves in $\Sigma$ that intersect $X$ efficiently.  Then $\Sigma - X$ is a $2g$-punctured sphere and $Y$ is a collection of properly embedded essential arcs in $\Sigma - X$.  Obtain a planar graph $\Gamma_X(Y)$ by having each puncture of $\Sigma - X$ correspond to a vertex and each arc of $Y$ in $\Sigma - X$ represent an edge.
\end{defin}

The graph $\Gamma_X(Y)$ is known as the \emph{Whitehead graph} (see \cite{Stallings}).

\begin{defin}
A graph $\Gamma$ is 2-connected if it is connected and $\Gamma$ does not contain a vertex whose removal would disconnect the graph (i.e. $\Gamma$ does not contain a cut vertex).
\end{defin}

The following theorems characterize diskbusting sets of curves on the surface of a handlebody and were first proven by Starr \cite{Starr}.  Alternate proofs have been given by Wu \cite{Wu} (Theorem 1.2), Strong \cite{Strong} (Theorem 3), and Luo \cite{Luo} (Theorem 3.1).

\begin{thm} \label{thm:Starr1}
\textnormal{(Starr \cite{Starr}, Theorem 1).} Let $H$ be a handlebody with outer boundary $\Sigma$.  Suppose $Y = \{y_1,y_2,...,y_t\}$ is a collection of essential simple closed curves in $\Sigma$.  Then $Y$ is diskbusting on $H$ if and only if there exists a standard set of meridians $X$ of $H$ such that $\Gamma_X(Y)$ is 2-connected.
\end{thm}

\begin{thm} \label{thm:Starr2}
\textnormal{(Starr \cite{Starr}, Theorem 2).} Let $H$ be a handlebody with outer boundary $\Sigma$.  Suppose $Y = \{y_1,y_2,...,y_t\}$ is a collection of essential simple closed curves in $\Sigma$.  Then $Y$ is diskbusting on $H$ if and only if there exists a pants decomposition $\mcP$ of $H$ such that $Y$ is collectively 1-seamed with respect to $\mcP$.
\end{thm}

Note that in the proof of Theorem \ref{thm:Starr2}, Starr proved that if $Y$ is a collection of essential simple closed curves such that $\Gamma_X(Y)$ is 2-connected (and therefore $Y$ is diskbusting), then there exists a pants decomposition $\mcP$ of $H$ such that $X \subset \mcP$ and $Y$ is collectively 1-seamed with respect to $\mcP$.

\begin{figure}[ht]
\labellist
\footnotesize \hair 2pt
\pinlabel \textcolor{red}{$\mcP$} at 85 32
\pinlabel \textcolor{blue}{$Y_1$} at 66 32
\pinlabel \textcolor{cyan}{$Y_2$} at 75 80
\endlabellist
\centering \scalebox{1.5}{\includegraphics{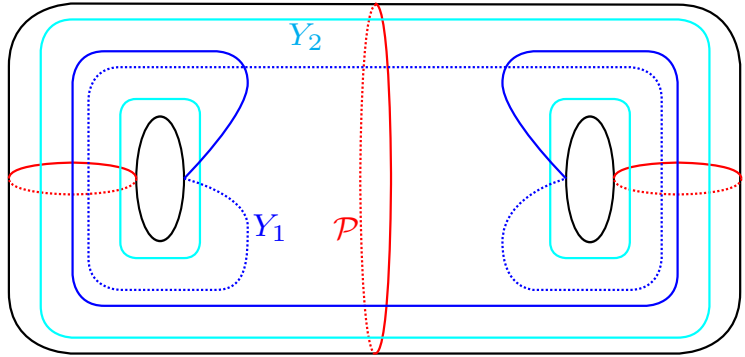}}
\caption{If $\mcP$ is a pants decomposition of $H$, then $Y_1$ and $Y_2$ are both diskbusting on $H$.}\label{fig:diskbusting}
\end{figure}

The following definition is a generalization of the notion of diskbusting to distance 3.  The term ``nearly fills'' was chosen since ``almost fills'' is already used by Hempel in \cite{He} for an alternate purpose.

\begin{defin}
Suppose $Y$ is a finite set of disjoint essential simple closed curves in the boundary $\Sigma$ of a handlebody $H$.  Then $Y$ \emph{nearly fills $H$} if for every meridian $d$ of $H$, the union of $d$ and $Y$ fills $\Sigma$.  That is, $\Sigma - (Y \cup d)$ is a collection of 2-cells.
\end{defin}
\begin{prop} \label{prop:nearlyfills}
Let $H$ be a handlebody with $\partial H = \Sigma$.  Then
\begin{enumerate}
\item A finite collection $Y$ of pairwise disjoint essential curves in $\Sigma$ nearly fills $H$ if and only if for every essential curve $\gamma$ that is disjoint from $Y$, $d(D(H),\gamma) \geq 2$.
\item A single simple closed curve $y$ in $\Sigma$ nearly fills $H$ if and only if $d(D(H),y) \geq 3$.
\end{enumerate}
\end{prop}
\begin{proof}
We begin with the proof of (1).  Suppose $Y$ nearly fills $H$.  Let $\beta$ be an essential curve disjoint from $Y$ and $d$ an arbitrary meridian of $H$.  If $\beta$ is disjoint from $d$, then $d \cup Y$ does not fill $\Sigma$, which contradicts our assumption.  Hence $\beta$ must intersect every meridian of $H$ and therefore $d(D(H),\gamma) \geq 2$.

Conversely, suppose $Y$ is a collection of pairwise disjoint essential simple closed curves with the property that any essential curve $\gamma$ disjoint from $Y$ satisfies $d(D(H),\gamma) \geq 2$.  Let $d$ be an arbitrary meridian of $H$.  If $d \cup Y$ does not fill $\Sigma$, then there exists some essential curve $\beta$ in $\Sigma$ that lies disjoint from $d$ and $Y$.  Since $d \in D(H)$, $d(D(H),\beta) < 2$ and we have a contradiction.

For the proof of (2), suppose $y$ is a single simple closed curve that nearly fills $H$.  By part (1), any curve $\gamma$ disjoint from $y$ satisfies $d(D(H),\gamma) \geq 2$.  Hence the length of any path between $y$ to $D(H)$ must be at least 3.
\end{proof}

\section{Main Theorems} \label{sec:main}

Suppose $H$ is a handlebody with $\partial H = \Sigma$ and consider the Heegaard splitting induced by the identity map on $\Sigma$.  The resulting 3-manifold is homeomorphic to $\#_g (S^2 \times S^1)$ and the Heegaard splitting has distance 0, as the disk sets of the two handlebodies are identical.  Hempel \cite {He} replaced the identity map with a high power of a pseudo-Anosov map to prove the existence of Heegaard splittings with arbitrarily high distance.  We give a similar result by providing a lower bound on the distance of the resulting Heegaard splitting when the identity map is replaced with a high power of a Dehn twist map.  As shown later in Lemma \ref{lem:upperbound}, Dehn twisting about a curve $\gamma$ that is distance $d$ from $D(H)$ determines a Heegaard splitting of distance at most $2d-2$.  We show that for a sufficiently high power of Dehn twists about $\gamma$ the resulting Heegaard splitting is exactly distance $2d-2$.

We first require the following utility lemma, which is a generalization of Lemma \ref{lem:disjointcircling}.

\begin{lem} \label{lem:twistchain}
Let $X = \{x_1,x_2,...,x_s\}$ be a collection of pairwise disjoint curves in $\Sigma$.  Suppose $y$ is a simple closed curve in $\Sigma$ with annular neighborhood $N$ that is fibered with respect to $X$.  Let $\alpha_0,\alpha_1,...,\alpha_n$ be a collection of essential simple closed curves such that each intersects $y$, each is in efficient position with respect to $(X,y,N)$, and $\alpha_j$ and $\alpha_{j+1}$ are disjoint for $0 \leq j \leq n-1$.  Assume a component of $\alpha_n \cap N$ has twisting number equal to $m_n$ with respect to $X$.  If $m_i$ is the twisting number of a component of $\alpha_i \cap N$ for some $0 \leq i \leq n-1$, then $|m_n - m_i| \leq |n-i|$.
\end{lem}
\begin{proof}
Let $b_n$ be the component of $\alpha_n \cap N$ that is assumed to have twisting number $m_n$.  For a fixed $i$ such that $0 \leq i \leq n-1$, let $m_i$ be the twisting number of a component of $\alpha_i \cap N$.  Then for $i < j < n$, there exists some component $b_j$ of $\alpha_j \cap N$ that is a properly embedded arc in $N$ since each $\alpha_j$ nontrivially intersects $y$.  Let $m_j$ denote the twisting number of $b_j$ in $N$ with respect to $X$.  Then by Lemma \ref{lem:disjointcircling}, as $b_{j+1}$ and $b_j$ are disjoint we have that $|m_{j+1} - m_j| \leq 1$ for each $j$.  So by the triangle inequality, $|m_n - m_i| \leq |n-i|$.
\end{proof}

We now prove an extension of a result of Casson and Gordon \cite{CG}, whose result is the following.  

\begin{thm} \label{thm:cg}
\textnormal{(Casson-Gordon \cite{CG}).} Let $H$ be a handlebody with $\partial H = \Sigma$ and $\gamma$ a simple closed curve such that $d(D(H),\gamma) \geq 2$.  Let $\mcP$ be a pants decomposition of $H$ (so $K_{\mcP} = D(H)$) such that $\gamma$ is 1-seamed with respect to $\mcP$.  Then for $k \geq 2$, $d(K_{\mcP},K_{\tau^k_{\gamma}(\mcP)}) \geq 2$.
\end{thm}

The above result of Casson and Gordon can be proved by showing that Dehn twisting at least twice about a 1-seamed curve will yield a Heegaard diagram that satisfies the Casson-Gordon rectangle condition (see Example 3 in \cite{Ko} for details).

Note that if $d(D(H),\gamma) \leq 1$, then there exists some meridian $\alpha$ in $H$ that is disjoint from $\gamma$, and therefore fixed by $\tau_{\gamma}$. Consequently, the Heegaard splitting induced by $\tau^k_{\gamma}$ will always have distance 0.   So if we are interested in creating Heegaard splittings of non-zero distance, we need to assume that $d(D(H),\gamma) \geq 2$.  Then by Theorem \ref{thm:Starr2}, there is a pants decomposition $\mcP$ of $H$ such that $\gamma$ is 1-seamed with respect to $\mcP$.

\begin{lem} \label{lem:singlecurve}
Let $H$ be a handlebody with $\partial H = \Sigma$ and $\gamma$ a simple closed curve such that $d(D(H),\gamma) \geq d$ for $d \geq 2$.  Let $\mcP$ be a pants decomposition of $H$ (so $K_{\mcP} = D(H)$) such that $\gamma$ is 1-seamed with respect to $\mcP$.  Then 
\begin{displaymath}
   d(K_{\mcP},K_{\tau^k_{\gamma}(\mcP)}) \geq \left\{
     \begin{array}{ll}
       $k$ & \mbox{if $2 \leq k \leq 2d-2$,}\\
       $2d-2$ & \mbox{if $k \geq 2d-2$.}
     \end{array}
   \right.
\end{displaymath} 
\end{lem}
\begin{proof}
To simplify notation, let $\mcP' = \tau_{\gamma}^{k}(\mcP)$.

Suppose, for the sake of a contradiction, that $d(K_{\mcP}, K_{\mcP'}) = \ell$ for some $\ell < k$ and $\ell < 2d-2$.  Then there exists a sequence of simple closed curves $\alpha_0, \alpha_1, ..., \alpha_{\ell}$ such that consecutive curves are disjoint, $\alpha_0 \in K_{\mcP}$, and $\alpha_{\ell} \in K_{\mcP'}$.  

Note that $d(K_{\mcP},\gamma) = d(K_{\mcP'},\gamma)$ since any path in $C(\Sigma)$ between $K_{\mcP}$ and $\gamma$ can be sent by $\tau^k_{\gamma}$ to a path between $K_{\mcP'}$ and $\gamma$ (since $\tau^k_{\gamma}$ fixes $\gamma$) and conversely, any path between $K_{\mcP'}$ and $\gamma$ can be sent to a path between $K_{\mcP}$ and $\gamma$ via $\tau^{-k}_{\gamma}$. Therefore $d(K_{\mcP},\gamma) \geq d$ implies that $d(K_{\mcP'},\gamma) \geq d$.

Since $\ell < 2d-2$, $j \leq d-2$ or $\ell-j \leq d-2$ for any $j$ such that $0 \leq j \leq \ell$.  Consequently, $\alpha_j$ is a distance at most $d-2$ from at least one of $\alpha_0$ or $\alpha_\ell$ in $C(\Sigma)$.  As $\alpha_0 \in K_{\mcP}$ and $\alpha_\ell \in K_{\mcP'}$, we have that $d(K_{\mcP},\alpha_j) \leq d-2$ or $d(K_{\mcP'}, \alpha_j) \leq d-2$.  Then $\gamma$ being at least distance $d$ from both $K_{\mcP}$ and $K_{\mcP'}$ implies that $\alpha_j$ must be at least distance 2 from $\gamma$.  Hence, $\alpha_j$ intersects $\gamma$ for each $0 \leq j \leq \ell$.

As $\alpha_\ell$ is a component of $K_{\mcP'}$, $\alpha_\ell = \tau^k_{\gamma}(\beta)$ for some $\beta \in K_{\mcP}$.  Lemma \ref{lem:pantswinding}(b) then implies $tw(\alpha_\ell,\gamma) \geq k$ with respect to $\mcP$.  Let $N$ be an annular neighborhood of $\gamma$ that is fibered with respect to $\mcP$ and assume $\alpha_0,\alpha_1,...,\alpha_\ell$ have all been isotoped so that they are each in efficient position with respect to $(\mcP,\gamma,N)$, intersect each other efficiently, and a component of $\alpha_\ell \cap N$ has twisting number $\geq k$ in $N$ with respect to $\mcP$.  Then by Lemma \ref{lem:twistchain}, each component of $\alpha_1 \cap N$ must have a twisting number of at least $k - (\ell - 1) \geq 2$ in $N$ with respect to $\mcP$. So $\alpha_1$ is 1-seamed with respect to $\mcP$ (since $\gamma$ is 1-seamed) and therefore $d(K_{\mcP},\alpha_1) \geq 2$ by Theorem \ref{thm:Starr2}.  But this is a contradiction as $d(K_{\mcP},\alpha_1) = 1$. Hence, $d(K_{\mcP},K_{\mcP'}) \geq \min\{k,2d-2\}$.
\end{proof}

One of the key arguments needed in the proof of Lemma \ref{lem:singlecurve} is confirming that, for any sequence $\alpha_0,\alpha_1,...,\alpha_{\ell}$ of simple closed curves that form a path in $C(\Sigma)$ between $K_{\mcP}$ and $K_{\tau^k_{\gamma}(\mcP)}$ of length less than $2d-2$, each $\alpha_j$ must intersect $\gamma$.  This allows us to use Lemma \ref{lem:twistchain} to show that the twisting numbers of each $\alpha_j$ about $\gamma$ are dependent on $k$.  In particular, their twisting numbers will increase as $k$ increases.  Therefore, we would expect that for sufficiently large values of $k$, the shortest path between $K_{\mcP}$ and $K_{\tau^k_{\gamma}(\mcP)}$ must include a curve that is disjoint from $\gamma$.  The following lemma shows it is not hard to construct a path of length $2d-2$ between $K_{\mcP}$ and $K_{\tau^k_{\gamma}(\mcP)}$ that indeed includes such a disjoint curve.

\begin{lem} \label{lem:upperbound}
Let $\mcP$ be a pants decomposition of $\Sigma$.  Suppose $\gamma$ is a simple closed curve such that $d(K_{\mcP},\gamma) = d$ for some $d \geq 1$.  Then for any $k \in \mathbb{Z}$, $d(K_{\mcP},K_{\tau^k_{\gamma}(\mcP)}) \leq 2d-2$.
\end{lem}
\begin{proof}
For the duration of this proof, let $\mcP' = \tau^k_{\gamma}(\mcP)$.

If $d(K_{\mcP},\gamma) = 1$, there exists some $\alpha \in K_{\mcP}$ that is disjoint from $\gamma$.  Then $\tau_{\gamma}$ fixes $\alpha$ and therefore $\alpha \in K_{\mcP'}$ and $d(K_{\mcP},K_{\mcP'}) = 0$.  

So suppose $d > 1$. By definition, $d(K_{\mcP},\gamma) = d$ implies there exists a sequence of simple closed curves $\alpha_0, \alpha_1, ..., \alpha_{d}$ such that $\alpha_j \cap \alpha_{j+1} = \emptyset$ for $0 \leq j < d$, $\alpha_0 \in K_{\mcP}$, and $\alpha_d = \gamma$.

As $\alpha_{d-1}$ is disjoint from both $\gamma$ and $\alpha_{d-2}$, $\alpha_{d-1}$ is also disjoint from $\tau^k_{\gamma}(\alpha_{d-2})$.  We then have the following path in $C(\Sigma)$ of $2d-2$ curves between $K_{\mcP}$ and $K_{\mcP'}$ (see Figure \ref{fig:path}):

$$\alpha_0, \alpha_1, ...,\alpha_{d-2}, \alpha_{d-1}, \tau^k_{\gamma}(\alpha_{d-2}),..., \tau^k_{\gamma}(\alpha_1), \tau^k_{\gamma}(\alpha_0).$$

\noindent Hence $d(K_{\mcP},K_{\mcP'}) \leq 2d-2$.
\end{proof}

\begin{figure}[ht]
\labellist
\footnotesize \hair 2pt
\pinlabel $\alpha_0$ at 2 14
\pinlabel $\alpha_1$ at  43 14
\pinlabel $\alpha_{d-2}$ at 82 14
\pinlabel $\alpha_{d-1}$ at 123 14
\pinlabel $\shortparallel$ at 123 8
\pinlabel $\tau^k_{\gamma}(\alpha_{d-1})$ at 123 1
\pinlabel $\gamma$ at 117 53
\pinlabel $\tau^k_{\gamma}(\alpha_{d-2})$ at 162 14
\pinlabel $\tau^k_{\gamma}(\alpha_1)$ at 207 14
\pinlabel $\tau^k_{\gamma}(\alpha_0)$ at 242 14
\pinlabel $K_{P}$ at 2 40
\pinlabel $K_{\tau^k_{\gamma}(\mcP)}$ at 242 40
\endlabellist
\centering \scalebox{1.4}{\includegraphics{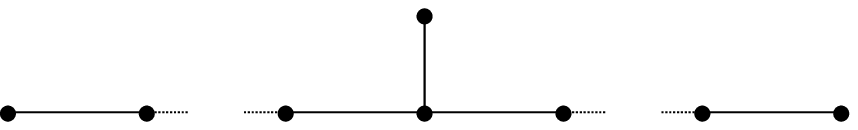}}
\caption{Path in $C(\Sigma)$ between $K_{\mcP}$ and $K_{\mcP'}$.}\label{fig:path}
\end{figure}

Combining the above two lemmas then gives the following result.

\begin{thm} \label{thm:exact}
Let $H$ be a handlebody with $\partial H = \Sigma$ and $\gamma$ a simple closed curve such that $d(D(H),\gamma) = d$ for $d \geq 2$.  Let $\mcP$ be a pants decomposition of $H$ (so $K_{\mcP} = D(H)$) such that $\gamma$ is 1-seamed with respect to $\mcP$.  Then for $k \geq 2d-2$, $d(K_{\mcP},K_{\tau^k_{\gamma}(\mcP)}) = 2d-2$.
\end{thm}
\begin{proof}
By Lemma \ref{lem:singlecurve}, for $k \geq 2d-2$, $d(K_{\mcP},K_{\tau^k_{\gamma}(\mcP)}) \geq 2d-2$.  On the other hand, Lemma  \ref{lem:upperbound} shows that Dehn twisting about a curve $\gamma$ such that $d(K_{\mcP},\gamma) = d$ implies that $d(K_{\mcP},K_{\tau^k_{\gamma}(\mcP)}) \leq 2d-2$.  Hence we achieve the desired equality.
\end{proof}

If we started with a simple closed curve $\gamma$ such that $d(D(H),\gamma) = 3$, the above theorem shows that for $k \geq 4$, $V_{\mcP} \cup_{\Sigma} V_{\tau^k_{\gamma}(\mcP)}$ is a 3-manifold that admits a Heegaard splitting of distance exactly 4.  While we lack the nice characterization of distance 3 curves that the work of Starr \cite{Starr} (Theorems \ref{thm:Starr1} and \ref{thm:Starr2}) and others provide for diskbusting curves, some distance $\geq 3$ Heegaard splitting criterions like those provided in \cite{Sc} and \cite{He} can be adapted to help identify distance 3 curves.  So suppose we have identified such a curve $\gamma$ and that $\alpha_0, \alpha_1, \alpha_2, \alpha_3$ are simple closed curves that form a minimum-length path in $C(\Sigma)$ such that $\alpha_0 \in D(H)$ and $\alpha_3 = \gamma$.  Then Theorem \ref{thm:exact} shows that Dehn twisting about $\gamma$ a total of 4 times will produce a distance 4 Heegaard splitting.  However, the arguments of Lemmas \ref{lem:singlecurve} and \ref{lem:upperbound} also show that $d(D(H),\tau^k_{\gamma}(\alpha_0)) = 4$ for $k \geq 4$.  So if we let $\beta = \tau^k_{\gamma}(\alpha_0)$ and consider Dehn twisting about $\beta$, we can produce a distance 6 Heegaard splitting and also $d(D(H),\tau^{k'}_{\beta}(\alpha_0)) = 6$ for $k' \geq 6$.  Repeating this process allows us to produce Heegaard splittings of arbitrarily high distance that have a known exact distance.

The above manifolds that admit these high distance Heegaard splittings can be considered the result of ($\frac{1}{k}$)-Dehn surgery along a curve $\gamma$ in the double of a handlebody.  Casson and Gordon \cite{CG} generalized their Theorem \ref{thm:cg} to the case where Dehn surgery is performed in a manifold admitting a weakly reducible (distance $\leq 1$) Heegaard splitting.  Their result can be restated as the following (see the appendix of \cite{MS} for a proof).

\begin{thm} \label{thm:cg2}
\textnormal{(Casson-Gordon \cite{CG}).} Let $H_1$ and $H_2$ be handlebodies such that $\partial H_1 = \partial H_2 = \Sigma$ and $d(D(H_1,D(H_2)) \leq 1$.  Suppose $\gamma$ is a simple closed curve such that $d(D(H_1,\gamma)) \geq 2$ and $d(D(H_2,\gamma)) \geq 2$.  Let $\mcP_1$ and $\mcP_2$ be pants decompositions of $H_1$ and $H_2$ respectively (so $K_{\mcP_1} = D(H_1)$ and $K_{\mcP_2} = D(H_2)$) such that $\gamma$ is 1-seamed with respect to both $\mcP_1$ and $\mcP_2$.  Then for $k \geq 6$, $d(K_{\mcP_1},K_{\tau^k_{\gamma}(\mcP_2)}) \geq 2$.
\end{thm}

We can extend this result of Casson and Gordon with the following lemma.

\begin{lem} \label{lem:arbmanifold}
Let $H_1$ and $H_2$ be handlebodies with $\partial H_1 = \partial H_2 = \Sigma$ and let $n = \max \{1,d(D(H_1),D(H_2))\}$. Suppose $d_1,d_2$ are integers such that $d_1,d_2 \geq 2$ and $d_1 + d_2 - 2 > n$ and let $\gamma$ be a simple closed curve such that $d(D(H_1),\gamma)) \geq d_1$ and $d(D(H_2),\gamma)) \geq d_2$.  Let $\mcP_1$ and $\mcP_2$ be pants decompositions of $H_1$ and $H_2$ respectively (so $K_{\mcP_1} = D(H_1)$ and $K_{\mcP_2} = D(H_2)$) such that $\gamma$ is 1-seamed with respect to both $\mcP_1$ and $\mcP_2$.  Then
\begin{displaymath}
   d(K_{\mcP_1},K_{\tau^{k+n+2}_{\gamma}(\mcP_2)}) \geq \left\{
     \begin{array}{ll}
       k & \mbox{if $2 \leq k \leq d_1 + d_2 - 2$,}\\
       d_1 + d_2 - 2 & \mbox{if $k \geq d_1 + d_2 - 2$.}
     \end{array}
   \right.
\end{displaymath}
\end{lem}
\begin{proof}
Assume that $\mcP_1$, $\mcP_2$, and $\gamma$ intersect efficiently.  Let $\bar{n} = d(D(H_1),D(H_2))$.  Then there exists a sequence of simple closed curves $\alpha_0$, ..., $\alpha_{\bar{n}}$ such that consecutive curves are disjoint, $\alpha_0 \in D(H_1)$, and $\alpha_{\bar{n}} \in D(H_2)$.  Suppose $N$ is an annular neighborhood of $\gamma$ that is fibered with respect to $\mcP_1$ and $\mcP_2$ and $\alpha_0$, ...,$\alpha_{\bar{n}}$ are all in efficient position with respect to $(\mcP_1,\gamma,N)$ and $(\mcP_2,\gamma,N)$ and also intersect each other efficiently.  Let $m_1,m_2,...,m_r$ denote the twisting numbers of the components of $\alpha_{\bar{n}} \cap N$ with respect to $\mcP_1$.

\vspace{10pt}

\noindent \textit{Claim 1:} $\max \{|m_1|,...,|m_r|\} \leq n$. 

\vspace{10pt}

\noindent \textit{Proof of Claim 1:} As $\bar{n} < d_1 + d_2 - 2$, either $j \leq d_1 - 2$ or $\bar{n}-j \leq d_2 - 2$ for all $j$ such that $0 \leq j \leq \bar{n}$.  If $j \leq d_1 - 2$, then $d(\alpha_j,D(H_1)) \leq d_1 - 2$.  On the other hand, $d(\gamma,D(H_1)) = d_1$, so it follows that $d(\alpha_j,\gamma) \geq 2$.  Similarly, the second case implies that $d(\alpha_j,D(H_2)) \leq d_2 - 2$ and therefore $d(\alpha_j,\gamma) \geq 2$.  In either case, $\alpha_j$ intersects $\gamma$ for each $0 \leq j \leq \bar{n}$.

Suppose, for the sake of a contradiction, that some component of $\alpha_{\bar{n}} \cap \gamma$ has twisting number $m^*$ with respect to $\mcP_1$ such that $|m^*| > n$. For $\bar{n} \geq 1$, Lemma \ref{lem:twistchain} implies the twisting numbers of the components of $\alpha_{1} \cap N$ will have absolute value strictly greater than $1$.  So $\alpha_{1}$ is diskbusting on $H_1$, which implies that $d(\alpha_1,D(H_1)) \geq 2$, a contradiction. If $\bar{n} = 0$, then $n = 1$ and $|m^*| > 1$ would imply that $\alpha_0$ is diskbusting, also a contradiction.  Hence $\max \{|m_1|,...,|m_r|\} \leq n$.  This concludes the proof of Claim 1.

\vspace{10pt}

Let $\beta$ be an arbitrary meridian of $H_2$.  Isotope $\beta$ to be in efficient position with respect to both $(\mcP_1,\gamma,N)$ and $(\mcP_2,\gamma,N)$, intersects $\alpha_{\bar{n}}$ efficiently, and let $b_1,b_2,...,b_s$ denote the twisting numbers of the components of $\beta \cap N$ with respect to $\mcP_1$.

\vspace{10pt}

\noindent \textit{Claim 2:} $\max \{|b_1|,...,|b_s|\} \leq n+2$.

\vspace{10pt}

\noindent \textit{Proof of Claim 2:} Lemma \ref{lem:compdiskwave} implies that any curve that bounds a disk in $H_2$ is either parallel to a component of $\mcP_2$ or contains a wave with respect to $\mcP_2$.  This implies that any meridian of $H_2$ that is in efficient position with respect to $(\mcP_2,\gamma,N)$ contains a component in $N$ that has twisting number 0 with respect to $\mcP_2$.  Therefore, as $\alpha_{\bar{n}}$ and $\beta$ are both meridians of $H_2$, there exists a component of $\alpha_{\bar{n}} \cap N$ and a component of $\beta \cap N$ that are disjoint in $N$.  So by Lemma \ref{lem:disjointcircling}, their twisting numbers must differ by at most 1.  Moreover, as the components of $\beta \cap N$ are all disjoint, $b_1,b_2,...,b_s$ must all be within 1 of each other.  It then follows from Claim 1 that $\max \{|b_1|,...,|b_s|\} \leq n+2$. This concludes the proof of Claim 2.

\vspace{10pt}

Claim 2 implies that the lower bound on the twisting number of any component of $\beta \cap N$ is $-(n+2)$.  So by Corollary \ref{cor:winding}, for $k \geq 2$, $tw(\tau^{k+n+2}_{\gamma}(\beta),\gamma) \geq k$.  As $\beta$ was an arbitrary meridian of $H_2$, this implies that every element $\beta' \in K_{\tau^{k+n+2}_{\gamma}(\mcP_2)}$ satisfies $tw(\beta',\gamma) \geq k$.

Suppose, for the sake of contradiction, that $\zeta_0, \zeta_1, ..., \zeta_t$ is a sequence of simple closed curves such that consecutive curves are disjoint, $\zeta_0 \in K_{\mcP_1}$, and $\zeta_t \in K_{\tau^k_\gamma(\mcP_2)}$ such that $t < d_1 + d_2 - 2$ and $t < k$.  Then $t < d_1 + d_2 - 2$ implies that $\zeta_{\ell}$ intersects $\gamma$ for each $0 \leq \ell \leq t$.  As $\zeta_t$ is an element of $K_{\tau^k_\gamma(\mcP_2)}$, we have that $tw(\zeta_t,\gamma) \geq k$.  So Lemma \ref{lem:twistchain} and $t < k$ implies that $tw(\zeta_1,\gamma) \geq 2$ and therefore $\zeta_1$ is diskbusting on $H_1$, which is a contradiction. Hence, $d(K_{\mcP_1}, K_{\tau^{k+n+2}_{\gamma}(\mcP_2)}) \geq \min \{k,d_1+d_2-2\}$.
\end{proof}

We can then combine the above lemma with an upper bound on the distance of the resulting splitting to show that Dehn surgery on a sufficiently complicated curve in a closed 3-manifold can produce a Heegaard splitting with lower and upper bounds on its distance that differ by 2.

\begin{thm} \label{thm:nots1xs2}
Let $H_1$ and $H_2$ be handlebodies with $\partial H_1 = \partial H_2 = \Sigma$ and let $n = \max \{1,d(D(H_1),D(H_2))\}$.  Suppose $d_1,d_2$ are integers such that $d_1,d_2 \geq 2$ and $d_1 + d_2 - 2 > n$ and let $\gamma$ be a simple closed curve such that $d(D(H_1),\gamma)) = d_1$ and $d(D(H_2),\gamma)) = d_2$.  Let $\mcP_1$ and $\mcP_2$ be pants decompositions of $H_1$ and $H_2$ respectively (so $K_{\mcP_1} = D(H_1)$ and $K_{\mcP_2} = D(H_2)$) such that $\gamma$ is 1-seamed with respect to both $\mcP_1$ and $\mcP_2$.  Then for $k \geq d_1 + d_2 - 2$, 
$$d_1 + d_2 - 2 \leq d(K_{\mcP_1},K_{\tau^{k+n+2}_{\gamma}(\mcP_2)}) \leq d_1 + d_2.$$
\end{thm}
\begin{proof}
The lower bound of $d_1+d_2-2$ is provided by Lemma \ref{lem:arbmanifold}.  As $d(\gamma,K_{\mcP_2}) = d_2$ implies that $d(\gamma,K_{\tau^{k+n+2}_{\gamma}(\mcP_2)}) = d_2$, we have
$$ d(K_{\mcP_1},K_{\tau^{k+n+2}_{\gamma}(\mcP_2)}) \leq d(K_{\mcP_1},\gamma) + d(\gamma,K_{\tau^{k+n+2}_{\gamma}(\mcP_2)}) \leq d_1 + d_2$$
by the triangle inequality.
\end{proof}

\section{Improving Evans' Result} \label{sec:evans}

We can also use Lemma \ref{lem:pantswinding} to simplify and improve a result by Evans \cite{Ev} that generates examples of high distance Heegaard splittings, which is presented below as Theorem \ref{thm:evans}.  Evans' approach uses an iterative process that carefully constructs a set of curves in the surface $\Sigma$ such that Dehn twisting twice about each of them produces a Heegaard splitting of high distance. This is in contrast to Lemma \ref{lem:singlecurve}, which only requires a single curve of high distance, but then needs a large number of Dehn twists to obtain a high distance splitting.

Note that the requirement for a ``$\gamma_s$-stack to have height at least 2'' in Evans' theorem is a similar requirement to $\gamma_s$ being 3-seamed with respect to $\mcP$.  We refer the reader to the work of Evans \cite{Ev} or Hempel \cite{He} for an introduction to stacks.  

\begin{thm}\label{thm:evans}
\textnormal{(\cite{Ev}, Theorem 4.4).} Let $X$ be a standard set of meridians of $\Sigma$.  Suppose $\gamma_s$ is a simple closed curve such that $\Gamma_X(\gamma_s)$ is 2-connected and each $\gamma_s$-stack has height at least 2.  For $n \geq 1$, let

\vspace{-10pt}

\begin{eqnarray*}
Y^1 = \tau^2_{\gamma_s}(X) && = \{y^1_1,... ,y^1_{g}\},\\
Y^2 = \tau^2_{Y^1}(X) && = \{y^2_1,...,y^2_{g}\},\\
Y^3 = \tau^2_{Y^2}(X) && = \{y^3_1,...,y^3_{g}\},\\
\vdots\\
Y^n = \tau^2_{Y^{n-1}}(X) && = \{y^n_1,...,y^n_{g}\}.
\end{eqnarray*}

Then $d(K_X,K_{Y^n}) \geq n$.
\end{thm}

We make the following improvements to Evans' result:  
\begin{itemize}
\item we eliminate the requirement for every $\gamma_s$-stack to have height at least 2,
\item except for the one iteration to obtain $Y^2$, we can restrict to Dehn twisting about a single curve, and
\item $d(K_X,K_{Y^n}) \geq (n+1)$ rather than $n$.
\end{itemize}

The full statement of the theorem then becomes the following.

\begin{thm}\label{thm:main}
Let $X$ be a standard set of meridians of $\Sigma$ and $\gamma_s$ a simple closed curve such that $\Gamma_X(\gamma_s)$ is 2-connected.  For $n \geq 1$ let

\vspace{-10pt}

\begin{eqnarray*}
Y^1 = \tau^2_{\gamma_s}(X) && = \{y^1_1,... ,y^1_{g}\},\\
Y^2 = \tau^2_{Y^1}(X) && = \{y^2_1,...,y^2_{g}\},\\
Y^3 = \tau^2_{y^2_{*}}(X) && = \{y^3_1,...,y^3_{g}\} \text{ for any }y^2_* \in Y^2,\\
\vdots \\
Y^n = \tau^2_{y^{n-1}_*}(X) && = \{y^n_1,...,y^n_{g}\} \text{ for any }y^{n-1}_* \in Y^{n-1}.
\end{eqnarray*}

Then $d(K_X,K_{Y^n}) \geq (n+1)$.
\end{thm}

Recall that Starr \cite{Starr} showed that if $\gamma_s$ is diskbusting on a handlebody, then there exists a standard set of meridians $X$ for the handlebody such that $\Gamma_X(\gamma_s)$ is 2-connected.  Therefore it suffices for $\gamma_s$ to be any diskbusting curve.  Moreover, Starr showed that if $\Gamma_X(\gamma_s)$ is 2-connected then we can extend $X$ to a pants decomposition $\mcP$ of $\Sigma$ such that $X \subset \mcP$ and $\gamma_s$ is 1-seamed with respect to $\mcP$.  Since $V_X \cong V_{\mcP}$, we will often use this $\mcP$ instead of $X$ in the following arguments if a pants decomposition is more useful.

For example, we will use this pants decomposition $\mcP$ in the following lemma, which is a simplified version of a lemma given in \cite{BTY} (see Lemma 3.5).  This ensures that for $i \geq 1$, each curve of $Y^i$ in Theorem \ref{thm:main} is 2-seamed with respect to $\mcP$.

\begin{lem} \label{lem:seamed}
Let $\mcP$ be a pants decomposition of $\Sigma$.  Then if $\rho$ is an element of $\mcP$ and $\gamma$ is an $m$-seamed curve with respect to $\mcP$, then $\tau^k_{\gamma}(\rho)$ is $2km^2$-seamed with respect to $\mcP$.
\end{lem}
\begin{proof}
As shown in Lemma \ref{lem:winding}, an embedding of $\tau^k_{\gamma}(\rho)$ can be obtained by taking an annular neighborhood $N$ of $\gamma$ that is fibered with respect to $\mcP$ and replacing subarcs of $\rho \cap N$ with arcs that have twisting number $k$ with respect to $\mcP$.  Each component of $\tau^k_{\gamma}(\rho) \cap N$ will then contain at least $km$ seams.  As $\gamma$ being $m$-seamed implies $\gamma$ intersects $\rho$ at least $2m$ times, $\tau^k_{\gamma}(\rho)$ is then $2km^2$-seamed with respect to $\mcP$.
\end{proof}

We now want to prove Theorem \ref{thm:main} for the cases when $n=1$ and $n=2$.  The case when $n=1$ is equivalent to the statement of Lemma \ref{lem:singlecurve} for $d=2$ and $k=2$.  As mentioned in the previous section, this result was first shown by Casson and Gordon in \cite{CG}.  So we begin by proving the case when $n=2$.

\begin{lem} \label{lem:nearlyfills}
Let $X$ be a standard set of meridians of $\Sigma$.  Suppose $\gamma$ is a simple closed curve such that $\Gamma_X(\gamma)$ is 2-connected.  For $k \geq 2$, $\tau_{\gamma}^k(X)$ nearly fills $V_X$.
\end{lem}
\begin{proof}
Let $\beta'$ be a simple closed curve in $\Sigma$ that is disjoint from $\tau_{\gamma}^k(X)$.  By Proposition \ref{prop:nearlyfills}, it then suffices to show that $\beta'$ must be at least distance 2 from $K_X$.

Theorem \ref{thm:Starr2} implies that we can obtain a pants decomposition $\mcP$ of $\Sigma$ such that $X \subset \mcP$ and $\gamma$ is 1-seamed with respect to $\mcP$.  As the Dehn twist operator preserves intersection number, we can let $\beta' = \tau_{\gamma}^k(\beta)$ for some curve $\beta$ that is disjoint from $X$.  Note that this implies $\beta$ bounds a disk in $V_X$.  Then by Lemma \ref{lem:pantswinding}, $tw(\beta',\gamma) \geq k$ with respect to $\mcP$.  Since $k \geq 2$ and $\gamma$ is 1-seamed, $\beta'$ is also 1-seamed with respect to $\mcP$ and therefore distance 2 from $K_X$ by Theorem \ref{thm:Starr2}.  Hence $\tau^k_{\gamma}(X)$ nearly fills $V_X$.
\end{proof}

The above lemma shows that Dehn twisting a standard set of meridians about a diskbusting curve yields a collection of curves that nearly fills.  In particular, this implies that the set $Y^1$ in Theorem \ref{thm:main} nearly fills $V_X$.

\begin{lem} \label{lem:dist3}
Let $\mcP$ be a pants decomposition of $\Sigma$.  Suppose $Y = \{y_1,y_2,...,y_t\}$ is a collection of curves in $\Sigma$ that nearly fills $V_{\mcP}$ and is collectively 2-seamed with respect to $\mcP$.  Then for $k \geq 2$, if $\beta'$ is a curve such that $d(K_{\tau^k_Y(\mcP)},\beta') = 1$, then $tw(\beta',y_j) > k-1$ for some $y_j \in Y$.
\end{lem}
\begin{proof}
As $d(K_{\tau^k_Y(\mcP)},\beta') = 1$, $\beta'$ is isotopic to $\tau^k_Y(\beta)$ for some curve $\beta$ that is distance 1 away from $K_{\mcP}$.  This has two implications.  The first is that $\beta$ is not diskbusting on $V_{\mcP}$ and therefore by Proposition \ref{prop:nearlyfills} must intersect some component $y_j \in Y$.  The second implication is that $\beta$ cannot be 1-seamed with respect to $\mcP$ by Theorem \ref{thm:Starr2}.  So Lemma \ref{lem:pantswinding} parts (a) and (c) together imply that $tw(\beta',y_j) > k-1$.
\end{proof}

The following corollary provides a more general version of a result by Hempel (see Theorem 5.4 in \cite{He}) to construct examples of splittings with distance at least 3.  In particular, Hempel requires Dehn twisting about a full set of curves, which is not necessary for our result.

\begin{cor} \label{cor:dist3}
Let $\mcP$ be a pants decomposition of $\Sigma$.  Suppose $Y = \{y_1,y_2,...,y_t\}$ is a collection of curves in $\Sigma$ that nearly fills $V_X$ such that each component $y_i \in Y$ is 1-seamed with respect to $\mcP$ and $Y$ is collectively 2-seamed with respect to $\mcP$.  Then for $k \geq 2$, $d(K_{\mcP},K_{\tau^k_Y(\mcP)}) \geq 3$.
\end{cor}
\begin{proof}
Let $\beta'$ be a curve such that $d(K_{\tau^k_Y(\mcP)},\beta') = 1$.  By Lemma \ref{lem:dist3}, $tw(\beta',y_j) > 1$ for some $y_j \in Y$.  As we have assumed that $y_j$ is 1-seamed with respect to $\mcP$, $\beta'$ must be 1-seamed as well.  Therefore, by Theorem \ref{thm:Starr2}, $d(K_{\mcP},\beta') \geq 2$ and so $d(K_{\mcP},K_{\tau^k_Y(\mcP)}) \geq 3$.
\end{proof}

We can now prove the case when $n=2$ in Theorem \ref{thm:main} in the following way.  Lemma \ref{lem:seamed} proves that each component of $Y^1$ is 1-seamed, which also means $Y^1$ is collectively 2-seamed since it has at least two components.  Lemma \ref{lem:nearlyfills} proves that $Y^1$ nearly fills $V_X$.  So we can apply Corollary \ref{cor:dist3} to $Y^1$ and conclude that $d(K_X,K_{Y^2}) \geq 3$.

Note that $Y^2$ is obtained by Dehn twisting about each component of $Y^1$, which has $g$ components.  If we wish to construct an example of a distance $\geq 3$ splitting by Dehn twisting about just a single curve, Corollary \ref{cor:dist3} implies that it is sufficient to Dehn twist twice about a curve that is distance 3 from $K_{\mcP}$ and 2-seamed with respect to $\mcP$.  Hempel \cite{He} and Berge \cite{Be} provide examples of distance 3 curves in a genus 2 surface that are 2-seamed with respect to an appropriate pants decomposition. 

We now provide the groundwork for an inductive argument to prove Theorem \ref{thm:main} for the case when $n > 2$.  Let $Y^i$ and $Y^{i-1}$ be defined as in Theorem \ref{thm:main} for $i \geq 2$. The next two results combine to show that if a curve $\beta'$ satisfies $tw(\beta',y^i_*) > 1$ for some element $y^i_*$ of $Y^i$, then any curve $\alpha'$ that is disjoint from $\beta'$ must satisfy $tw(\alpha',y^{i-1}_*) > 1$ for some $y^{i-1}_*$ of $Y^{i-1}$.

\begin{lem} \label{lem:doublestar}
Let $\mcP$ be a pants decomposition of $\Sigma$ .  Assume $Y = \{y_1,y_2,...,y_t\}$ is a collection of pairwise disjoint simple closed curves such that $Y$ nearly fills $V_{\mcP}$.  Moreover, let $y'_m = \tau^2_Y(x_m)$ for some $x_m \in \mcP$.  Suppose $\beta'$ is a simple closed curve such that $tw(\beta',y'_m) > 1$ with respect to $\mcP$.  If $\alpha'$ is a simple closed curve disjoint from $\beta'$, then $\alpha'$ must intersect some component $y_j \in Y$.
\end{lem}
\begin{proof}
Let $C = \{C_1,C_2,...,C_t\}$ be a collection of disjoint annular neighborhoods of the components of $Y$ that are each fibered with respect to $\mcP$.  By Lemma \ref{lem:winding}, we can obtain an embedding of $y'_m = \tau^2_Y(x_m)$ that is in efficient position with respect to $(\mcP,Y,C)$.  Note that each point of $y'_m \cap \mcP$ lies in $C$.  

\begin{figure}[ht]
\labellist
\footnotesize \hair 2pt
\pinlabel \textcolor{red}{$x_m$} at 11 0
\pinlabel \textcolor{red}{$x_m$} at 59 0
\pinlabel \textcolor{red}{$x_m$} at 167 0
\pinlabel \textcolor{blue}{$Y$} at 24 20
\pinlabel \textcolor{blue}{$Y$} at 24 61
\pinlabel \textcolor{blue}{$Y$} at 100 20
\pinlabel \textcolor{blue}{$Y$} at 100 61
\pinlabel \textcolor{blue}{$Y$} at 208 20
\pinlabel \textcolor{blue}{$Y$} at 208 61
\pinlabel \textcolor{green}{$y'_m$} at 79 1
\pinlabel \textcolor{green}{$y'_m$} at 191 1
\pinlabel $C$ at 16 12
\pinlabel $C$ at 16 53
\pinlabel $C$ at 92 12
\pinlabel $C$ at 92 53
\pinlabel $C$ at 200 12
\pinlabel $C$ at 200 53
\pinlabel $N$ at 143 15
\pinlabel $N$ at 143 56
\endlabellist
\centering \scalebox{1.5}{\includegraphics{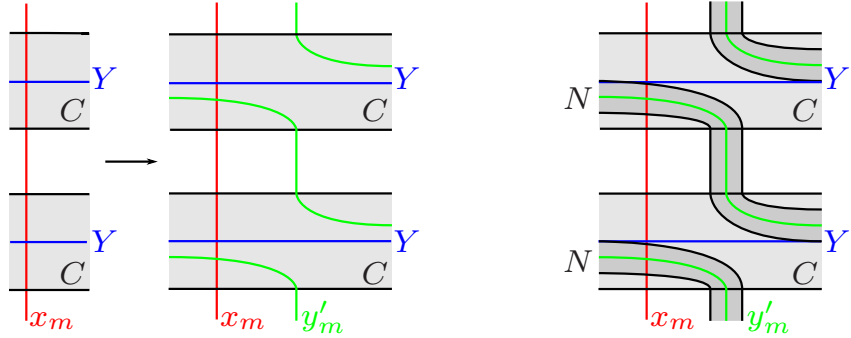}}
\caption{$N$ is chosen so that all components of $N \cap \mcP$ lie in $C$.}\label{fig:doublestar1}
\end{figure}

Now choose an annular neighborhood $N$ of $y'_m$ that is fibered with respect to $\mcP$ such that each component of $N \cap \mcP$ lies entirely in $C$ (see Figure \ref{fig:doublestar1}).  By hypothesis, there exists an embedding of $\beta'$ that intersects $(\mcP,y'_m,N)$ efficiently and a subarc $b$ of $\beta' \cap N$ has twisting number strictly greater than 1 with respect to $\mcP$.  So $b$ intersects some component $\rho$ of $\mcP \cap C_j$ at least twice for some $C_j \in C$.  Observe that any curve that intersects $y'_m$ nontrivally must also intersect $b$ or $\rho$ nontrivially (see Figure \ref{fig:doublestar2}).

\begin{figure}[ht]
\labellist
\footnotesize \hair 2pt
\pinlabel \textcolor{red}{$\rho$} at 43 100
\pinlabel \textcolor{green}{$y'_m$} at 8 40
\pinlabel $N$ at 90 15
\pinlabel \textcolor{magenta}{$b$} at 62 75
\endlabellist
\centering \scalebox{1.5}{\includegraphics{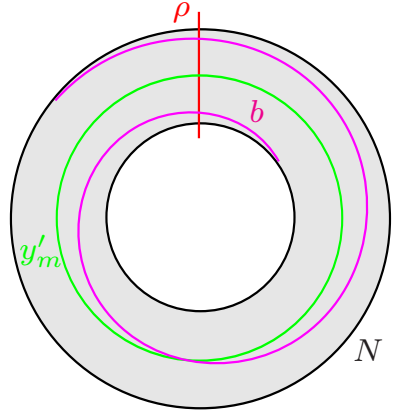}}
\caption{A curve that intersects $y'_m$ nontrivially must also intersect either $\rho$ or $b$.}\label{fig:doublestar2}
\end{figure}

Suppose, for the sake of contradiction, that $\alpha'$ is disjoint from $Y$.  Then as $Y$ nearly fills $V_X$, by definition $\alpha'$ must be at least distance 2 from $K_X$.  So $\alpha'$ must intersect the meridian $x_m$ and, as $\alpha'$ is fixed by $\tau_Y$, it also intersects $y'_m$.  Thus $\alpha'$ intersects either $b$ or $\rho$ nontrivially.  In the former case, $\alpha'$ then intersects $\beta'$, a contradiction of our hypothesis.  In the latter case, $\alpha'$ then has nontrivial intersection with $C_j$, as otherwise we could have isotopied $\alpha'$ to be disjoint from $\rho$.  This implies $\alpha'$ intersects $y_j \in Y$, a contradiction.
\end{proof}

The following lemma is based on a result in \cite{Ev} (see Lemma 4.3), which was later restated with a new proof in \cite{BTY} (see Lemma 3.4).  We provide a modified statement and proof since we desire twisting numbers that are strictly greater than 1.

\begin{lem} \label{lem:disjointwinding}
Let $\mcP$ be a pants decomposition of $\Sigma$.  Assume $Y = \{y_1,y_2,...,y_t\}$ is a collection of pairwise disjoint simple closed curves that intersect $\mcP$ efficiently, each curve is 1-seamed with respect to $\mcP$, and $Y$ is collectively 2-seamed.  Moreover, let $y'_m = \tau^2_{Y}(x_m)$ for some $x_m \in \mcP$.  Suppose $\beta'$ is a simple closed curve such that $tw(\beta',y'_m) > 1$ with respect to $\mcP$.  If $\alpha'$ is disjoint from $\beta'$ and $\alpha'$ intersects some $y_j \in Y$, then $tw(\alpha',y_j) > 1$.
\end{lem}
\begin{proof}
Let $C$ be an annular neighborhood of $y_j$ that is fibered with respect to $\mcP$. Assume we take a copy of $x_m$ that has been pushed off of $\mcP$.  Then the arcs of $x_m \cap C$ all have twisting number 0.  Moreover, there are at least two components of $x_m \cap C$ since $y_j$ is 1-seamed with respect to $\mcP$.  Lemma \ref{lem:winding} implies that we can obtain an embedding of $y'_m$ that is in efficient position with respect to $(\mcP,y_j,C)$ such that $y'_m \cap C$ contains at least two arcs with twisting number 2 that we will denote as $c_1$ and $c_2$.  Since $Y$ is collectively 2-seamed, we can then use Lemma \ref{lem:pantswinding} part (c) to show that the twisting numbers of $c_1$ and $c_2$ are strictly greater than 2. Note that we can assume all intersections between $y'_m$ and $\mcP$ lie in $C$.

Now let $N$ be an annular neighborhood of $y'_m$ that is fibered with respect to $\mcP$ and each component of $N \cap \mcP$ lies entirely in $C$.  As $tw(\beta',y'_m) > 1$, there exists an embedding of $\beta'$ that is in efficient position with respect to $(\mcP,y'_m,N)$ and a component $b$ of $\beta' \cap N$ has twisting number greater than 1.  We can therefore isotope $\beta'$ within $N$ to agree with $y'_m$ except for a small neighborhood around a point of intersection $r$ between $y'_m$ and $\mcP$ (see Figure \ref{fig:isotopyoverlap}).  As $r$ lies in $C$ we can assume this small neighborhood is contained in $C$.  Since $c_1$ and $c_2$ are distinct components of $y'_m \cap C$, $b$ coincides with at least one of them.  So $tw(\beta',y_j) > 2$.  As $\alpha'$ is assumed to be a curve that intersects $y_j$ and is disjoint from $\beta'$, by Lemma \ref{lem:disjointcircling} we have $tw(\alpha',y_j) > 1$.
\end{proof}

\begin{figure}[ht]
\labellist
\footnotesize \hair 2pt
\pinlabel \textcolor{red}{$\mcP$} at 49 103
\pinlabel \textcolor{red}{$\mcP$} at 175 103
\pinlabel \textcolor{magenta}{$b$} at 67 9
\pinlabel \textcolor{magenta}{$b$} at 193 12
\pinlabel \textcolor{blue}{$y'_m$} at 23 20
\pinlabel $N$ at 85 7
\pinlabel $N$ at 211 7
\pinlabel $r$ at 75 105
\pinlabel $r$ at 199 105
\endlabellist
\centering \scalebox{1.5}{\includegraphics{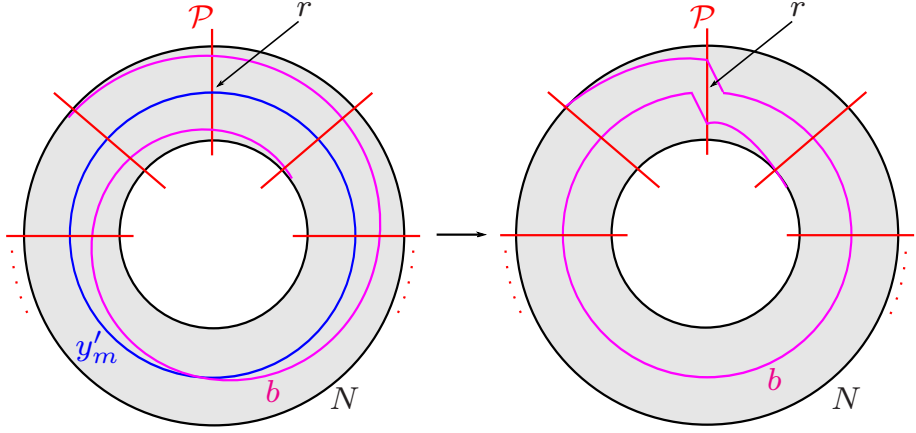}}
\caption{$b$ is isotoped to coincide with $y'_m$ everywhere except a small neighborhood of $r$.}\label{fig:isotopyoverlap}
\end{figure}

We can now prove Theorem \ref{thm:main}.

\vspace{10pt}

\noindent \textit{Proof of Theorem \ref{thm:main}.} 
Extend $X$ to a pants decomposition $\mcP$ of $\Sigma$ such that $X \subset \mcP$ and $\gamma_s$ is 1-seamed with respect to $\mcP$.

As mentioned earlier, the case when $n=1$ is a special case of Theorem \ref{thm:exact} (when $d = 2$ and $k = 2$).  For the case when $n=2$, we use Lemma \ref{lem:seamed} to show that each component of $Y^i$ for $i \geq 1$ is 2-seamed with respect to $\mcP$ (this fact is also used for the case when $n > 2$).  Moreover, this implies each $Y^i$ is collectively 2-seamed since $Y^i$ contains at least $g$ components where $g \geq 2$.  As Lemma \ref{lem:nearlyfills} shows that $Y^1$ nearly fills $V_X$, $Y^1$ satisfies the hypotheses of Corollary \ref{cor:dist3}, which shows that $d(K_X,K_{Y^2}) \geq 3$.

So assume inductively that $d(K_X,K_{Y^i}) \geq i+1$ for $i < n$ and $n \geq 3$ and it remains to prove that $d(K_X,K_{Y^n}) \geq n+1$.  In particular, this means that $Y^n$ is the image of $X$ after two Dehn twists around a single curve $y^{n-1}_*$ that is distance at least 3 from $K_X$ and is 2-seamed with respect to $\mcP$ by Lemma \ref{lem:seamed}.

In search of a contradiction, suppose that $d(K_X,K_{Y^n}) = m$ for some $2 \leq m \leq n$.  Then there exists a sequence of curves $\alpha_0, \alpha_1, ..., \alpha_m$ such that $\alpha_j \cap \alpha_{j+1} = \emptyset$ for $0 \leq j < m$, $\alpha_0 \in K_X$, and $\alpha_m \in K_{Y^n}$.  

As $n \geq 3$, $d(K_X,y^{n-1}) \geq 3$ by our inductive hypothesis and therefore $y^{n-1}$ nearly fills $V_X$.  Then since $\alpha_{m-1}$ is distance 1 from $K_{Y^n}$, we can apply Lemma \ref{lem:dist3} to get $tw(\alpha_{m-1},y^{n-1}_*) > 1$.  For $i \geq 2$, each component of $Y_i$ nearly fills $V_X$ and is 2-seamed with respect to $\mcP$ by Lemma \ref{lem:seamed}.  If $m > 2$, we can then apply Lemma \ref{lem:doublestar} to show that a curve disjoint from $\alpha_{m-1}$ (in particular, $\alpha_{m-2}$), must intersect $y^{n-2}_*$ and then Lemma \ref{lem:disjointwinding} implies that $tw(\alpha_{m-2},y^{n-2}_*) > 1$.  We can then repeat the application of Lemma \ref{lem:doublestar} and Lemma \ref{lem:disjointwinding} and conclude that $tw(\alpha_1,y^{n-m+1}_*) > 1$ where $y^{n-m+1}_*$ is an element of $Y^{n-m+1}$.  Observe that in the case when $(n-m+1) > 1$, $y^{n-m+1}_*$ is precisely the element of $Y^{n-m+1}$ used to generate $Y^{n-m+2}$ as specified in the statement of the theorem.  If $(n-m+1) = 1$, then since $Y^2$ was obtained by Dehn twisting about all $g$ curves of $Y^1$, $y^1_*$ represents some component of $Y^1$ that we know had to intersect $\alpha_1$ by Lemma \ref{lem:doublestar}.

In either case, $\alpha_1$ has a subarc with twisting number strictly greater than 1 around a 1-seamed curve.  So $\alpha_1$ is also 1-seamed with respect to $\mcP$ and $d(K_X,\alpha_1) \geq 2$ by Theorem \ref{thm:Starr2}.  But we had assumed that $\alpha_1$ was distance at most 1 from $K_X$ and therefore we have a contradiction.

Hence $d(K_X,K_{Y^n}) \geq (n+1)$. \hfill $\square$

\vspace{10pt}

The following are some observations about the examples generated in Theorem \ref{thm:main}.

\vspace{5pt}

\textit{1. The manifolds $V_X \cup_{\Sigma} V_{Y^n}$ are Haken.} 

\vspace{5pt}

Evans (see Theorem 3.4 and Corollary 3.5 in \cite{Ev}) showed that given a collection of pairwise disjoint simple closed curves $Z$ and $Z' = \tau^k_Z(X)$ where $k > 0$, then the resulting manifold determined by $V_X \cup_{\Sigma} V_{Z'}$ is Haken if $Z$ is not full (for example, if $Z$ contained less than $g$ components).  Moreover, she showed that $\pi(V_X \cup_{\Sigma} V_{Z'}) \rightarrow \pi(V_X \cup_{\Sigma} V_{Z})$ is a surjection.  When applied to Theorem \ref{thm:main} (for $n \neq 2$), the manifold determined by $V_X \cup_{\Sigma} V_{Y^n}$ is Haken since it was obtained by Dehn twisting about a single curve.  $V_X \cup_{\Sigma} V_{Y^2}$ is also Haken since its fundamental group surjects onto $V_X \cup_{\Sigma} V_{Y^1}$, which was Haken.  

The same argument can be used to show that the manifolds resulting from Theorem \ref{thm:exact} are also Haken.  Recently, Li \cite{Li} has produced examples of high distance non-Haken Heegaard splittings.


\vspace{5pt}

\textit{2. For $n \geq 2g$, the manifold determined by $V_X \cup_{\Sigma} V_{Y^n}$ has minimal genus $g$.}

\vspace{5pt}

Scharlemann and Tomova (see Corollary 4.5 in \cite{ST}) showed that if the distance $d$ of a genus $g$ Heegaard splitting with splitting surface $\Sigma$ is strictly greater than $2g$, then the splitting is minimal genus.  Since Theorem \ref{thm:main} shows that $V_X \cup_{\Sigma} V_{Y^n}$ has distance at least $n+1$, if $n \geq 2g$, then the splitting $V_X \cup_{\Sigma} V_{Y^n}$ is a minimal genus splitting.

\newpage
\bibliographystyle{amsplain}
\bibliography{high-distance-11-30}

\providecommand{\bysame}{\leavevmode\hbox to3em{\hrulefill}\thinspace}
\providecommand{\MR}{\relax\ifhmode\unskip\space\fi MR }
\providecommand{\MRhref}[2]{%
  \href{http://www.ams.org/mathscinet-getitem?mr=#1}{#2}
}
\providecommand{\href}[2]{#2}
\begin{thebibliography}{10}

\bibitem{Be}
John Berge, \emph{A closed orientable 3-manifold with distinct distance three
  genus two heegaard splittings}, 2009.

\bibitem{BTY}
Ryan Blair, Maggy Tomova, and Michael Yoshizawa, \emph{High distance bridge
  surfaces}, 2012.

\bibitem{CG}
Andrew Casson and C.~McA. Gordon, \emph{Manifolds with irreducible {H}eegaard
  splittings of arbitrarily high genus}, (unpublished) (1987).

\bibitem{Ev}
Tatiana Evans, \emph{High distance {H}eegaard splittings of 3-manifolds},
  Topology Appl. \textbf{153} (2006), no.~14, 2631--2647. \MR{2243739
  (2007j:57020)}

\bibitem{FM}
Benson Farb and Dan Margalit, \emph{A primer on mapping class groups},
  Princeton University Press, 2011, Version 5.0.

\bibitem{Ha}
Kevin Hartshorn, \emph{Heegaard splittings of {H}aken manifolds have bounded
  distance}, Pacific J. Math. \textbf{204} (2002), no.~1, 61--75. \MR{1905192
  (2003a:57037)}

\bibitem{HS}
Joel Hass and Peter Scott, \emph{Shortening curves on surfaces}, Topology
  \textbf{33} (1994), no.~1, 25--43. \MR{1259513 (94k:57025)}

\bibitem{He}
John Hempel, \emph{3-manifolds as viewed from the curve complex}, Topology
  \textbf{40} (2001), no.~3, 631--657. \MR{1838999 (2002f:57044)}

\bibitem{IJK}
Ayako Ido, Yeonhee Jang, and Tsuyoshi Kobayashi, \emph{Heegaard splittings of
  distance exactly $n$}, 2012.

\bibitem{Ko}
Tsuyoshi Kobayashi, \emph{Casson-{G}ordon's rectangle condition of {H}eegaard
  diagrams and incompressible tori in {$3$}-manifolds}, Osaka J. Math.
  \textbf{25} (1988), no.~3, 553--573. \MR{969018 (89m:57016)}

\bibitem{Ko2}
\bysame, \emph{Heights of simple loops and pseudo-{A}nosov homeomorphisms},
  Braids ({S}anta {C}ruz, {CA}, 1986), Contemp. Math., vol.~78, Amer. Math.
  Soc., Providence, RI, 1988, pp.~327--338. \MR{975087 (89m:57015)}

\bibitem{Li}
Tao Li, \emph{Heegaard splittings of distance exactly $n$}, 2012.

\bibitem{Luo}
Feng Luo, \emph{Heegaard diagrams and handlebody groups}, Topology Appl.
  \textbf{129} (2003), no.~2, 111--127. \MR{1961393 (2004a:57030)}

\bibitem{LM}
Martin Lustig and Yoav Moriah, \emph{High distance {H}eegaard splittings via
  fat train tracks}, Topology Appl. \textbf{156} (2009), no.~6, 1118--1129.
  \MR{2493372 (2011b:57025)}

\bibitem{MM3}
Howard~A. Masur and Yair~N. Minsky, \emph{Quasiconvexity in the curve complex},
  In the tradition of {A}hlfors and {B}ers, {III}, Contemp. Math., vol. 355,
  Amer. Math. Soc., Providence, RI, 2004, pp.~309--320. \MR{2145071
  (2006a:57022)}

\bibitem{MS}
Yoav Moriah and Jennifer Schultens, \emph{Irreducible {H}eegaard splittings of
  {S}eifert fibered spaces are either vertical or horizontal}, Topology
  \textbf{37} (1998), no.~5, 1089--1112. \MR{1650355 (99g:57021)}

\bibitem{Sc}
Martin Scharlemann, \emph{Berge's distance 3 pairs of genus 2 {H}eegaard
  splittings}, Math. Proc. Cambridge Philos. Soc. \textbf{151} (2011), no.~2,
  293--306. \MR{2823137 (2012g:57039)}

\bibitem{ST}
Martin Scharlemann and Maggy Tomova, \emph{Alternate {H}eegaard genus bounds
  distance}, Geom. Topol. \textbf{10} (2006), 593--617 (electronic).
  \MR{2224466 (2007b:57040)}

\bibitem{Stallings}
John~R. Stallings, \emph{Whitehead graphs on handlebodies}, Geometric group
  theory down under ({C}anberra, 1996), de Gruyter, Berlin, 1999, pp.~317--330.
  \MR{1714852 (2001i:57028)}

\bibitem{Starr}
Edith~Nelson Starr, \emph{Curves in handlebodies}, ProQuest LLC, Ann Arbor, MI,
  1992, Thesis (Ph.D.)--University of California, Berkeley. \MR{2688492}

\bibitem{Strong}
Richard Strong, \emph{Diskbusting elements of the free group}, Mathematical
  Research Letters \textbf{4} (1997), no.~2, 201--210.

\bibitem{Th}
Abigail Thompson, \emph{The disjoint curve property and genus {$2$} manifolds},
  Topology Appl. \textbf{97} (1999), no.~3, 273--279. \MR{1711418
  (2000h:57015)}

\bibitem{Wu}
Ying-Qing Wu, \emph{Incompressible surfaces and {D}ehn surgery on {$1$}-bridge
  knots in handlebodies}, Math. Proc. Cambridge Philos. Soc. \textbf{120}
  (1996), no.~4, 687--696. \MR{1401956 (97i:57021)}

\end{thebibliography}

\end{document}